\documentclass{article}
\usepackage{authblk}
\usepackage[margin=1in]{geometry}
\usepackage{subfigure}
\usepackage{cellspace}%
\usepackage{makecell}
\setcellgapes{3pt}

\usepackage{tikz}
\usetikzlibrary{matrix} 
\usetikzlibrary{arrows,automata}
\usetikzlibrary{positioning}

\usepackage{amssymb,mathrsfs,amsmath,amsthm,bm,diagbox,float}
\usepackage{multirow}
\usetikzlibrary{matrix,arrows}

\usepackage{hyperref}

\usepackage{mathtools}

\DeclarePairedDelimiter{\floor}{\lfloor}{\rfloor}

\renewcommand{\S}{\mathcal{S}}
\newcommand{\C}{\mathcal{C}}
\newcommand{\A}{\mathcal{A}}

\newtheorem{theorem}{Theorem}[section]
\newtheorem{definition}[theorem]{Definition}

\newtheorem{corollary}[theorem]{Corollary}
\newtheorem{lemma}[theorem]{Lemma}
\theoremstyle{definition}
\newtheorem{example}[theorem]{Example}

\title{Cyclic permutations avoiding patterns in both one-line and cycle forms}
\author[1]{Kassie Archer}
\author[2]{Ethan Borsh}
\author[3]{Jensen Bridges}
\author[4]{Christina Graves}
\author[4]{Millie Jeske}
\affil[1]{\footnotesize{Department of Mathematics, United States Naval Academy, Annapolis, MD, 21402}}
\affil[2]{\footnotesize{Mathematics Department, Allegheny College, Meadville, PA, 16335}}
\affil[3]{\footnotesize{Mathematics Department, Oklahoma State University, Stillwater, OK, 74074}}
\affil[4]{\footnotesize{Department of Mathematics, University of Texas at Tyler, Tyler, TX, 75799}}
\date{}

\begin{document}

\maketitle

\begin{abstract}
In this paper, we enumerate the set of cyclic permutations in $\mathcal{S}_n$ that classically avoid $\sigma\in\S_3$ in their one-line notation and avoid another pattern $\tau\in\S_3$ in their standard cycle notation. We find results for all pairs of patterns $(\sigma,\tau)$ in terms of Fibonacci numbers, binomial coefficients, and polynomial expressions. 
\end{abstract}

\section{Introduction}

Pattern avoidance is a notion typically defined for the one-line notation of a permutation. We say that a permutation $\pi = \pi_1\pi_2\ldots \pi_n$ contains another permutation (or pattern) $\sigma=\sigma_1\sigma_2\ldots \sigma_k$ if there is some subsequence of elements in the one-line form of $\pi$ that is in the same relative order as the one-line form of $\sigma$; we say $\pi$ avoids $\sigma$ if it does not contain it. 

It is still an open question to enumerate cyclic permutations that avoid a given pattern in their one-line notation. Generally, it has proven difficult to answer questions about cycle type or other algebraic properties of pattern-avoiding permutations, though some interesting results have been proven. These results include problems related to strong avoidance (meaning that both a permutation and its square avoid a given pattern) \cite{BS19,BD19,P23}, pattern avoidance and the group structure of $\S_n$ \cite{BD19},  pattern-avoiding permutations composed only of cycles of a certain size \cite{AG21, DRS07,GM02,SS85}, and cyclic permutations that avoid sets of patterns \cite{A22,AE14,AL20,BC19,H19} or consecutive patterns \cite{ET18,S07}. 

Pattern avoidance within the cycle notation itself has also been considered previously. In \cite{DELMSSS,V03}, the authors consider the notion of cyclic pattern avoidance, where all cyclic rotations of a permutation must avoid a given pattern. 
In \cite{AL17,E11}, the authors investigate the cycle type of almost-increasing permutations, characterizing these permutations in terms of pattern avoidance. A related concept of Boolean permutations (see for example, \cite{T22}) which avoid the pair 321 and 3412 in its one-line notation are characterized each cycle in the cycle form (beginning with its smallest element) avoiding the patterns 213 and 312. Interestingly, shallow permutations, defined to be those which have minimal depth for a given length and reflection length, have also been characterized as avoiding certain vincular patterns in their cycle form \cite{BT23}. 

In this paper, we also consider cyclic permutations $\pi = \pi_1\pi_2\ldots\pi_n = (1, c_2, c_3, \ldots, c_n)$ avoiding a pattern in its one-line form and its cycle form. In particular, we enumerate the set of cyclic permutations that simultaneously avoid a given pattern $\sigma\in\S_3$ in the one-line notation and avoid another pattern $\tau\in\S_3$ in the standard cycle notation. 
The nontrivial results are summarized by the chart in Figure~\ref{figtable}.

\begin{figure}
\centering\makegapedcells
\renewcommand{\arraystretch}{1.6}
\begin{tabular}{|c|c|c|c|c|}
\hline
$\sigma$ & $\tau$ & $a_n(\sigma;\tau)$ & Theorem & OEIS \\ \hline \hline
123 & \multirow{ 6}{*}{$213$}  & $\displaystyle \left\lceil \frac{n}{2}\right\rceil +1$ & Theorem~\ref{thm:123-213}& A004526\\ \cline{3-5} \cline{1-1} 
 132 & &$n-1$& Theorem~\ref{thm:132-213}& A000027\\ \cline{3-5} \cline{1-1}
 213 & &$F_n$& Theorem~\ref{thm:213-213}& A000045\\ \cline{3-5} \cline{1-1}
 231 &&$\displaystyle \sum_{k=0}^{\lfloor\frac{n-1}{3}\rfloor}\binom{n-1-k}{2k}$ & Theorem~\ref{thm:231-213}& A005251 \\ \cline{3-5} \cline{1-1}
 312 & & 1& Theorem~\ref{thm:312-213}& A000012 \\ \cline{3-5} \cline{1-1}
 321 & & $2^{n-2}$& Theorem~\ref{thm:321-213}& A000079 \\ \hline \hline
123 &  \multirow{ 8}{*}{$231$} & $\displaystyle 5{\lfloor\frac{n-1}{2}\rfloor \choose 2} +\begin{cases} \displaystyle{ 2n-5} & \text{if } n \text{ is even},\\ 
\displaystyle{n-2} & \text{if } n \text{ is odd}. \end{cases}$& Theorem~\ref{thm:123-231}& $\begin{cases} \text{A326725} & \text{if } n \text{ is even,}\\ \text{A140066}-2 & \text{if } n \text{ is odd.} \end{cases}$\\ \cline{3-5} \cline{1-1} 
132 & &$2F_n-2$&Theorem~\ref{thm:132-231} & A019274\\ \cline{3-5} \cline{1-1} 
213 & &$\displaystyle \left\lfloor \frac{n^2}{4}\right\rfloor$& Theorem~\ref{thm:213-231}& A002620\\ \cline{3-5} \cline{1-1} 
231 & &$n$ & Theorem~\ref{thm:231-231}& A000027\\ \cline{3-5} \cline{1-1} 
 312 &&$ \displaystyle \sum_{k=0}^{\lfloor\frac{n-1}{3}\rfloor}\binom{n-1-k}{2k}$  & Theorem~\ref{thm:312-231}& A005251\\\cline{3-5} \cline{1-1} 
 321 && $n-1$& Theorem~\ref{thm:321-231}& A000027\\ \hline \hline
123 & \multirow{ 9}{*}{$321$}  & eventually 0  & Theorem~\ref{thm:123-321}& A000004\\ \cline{3-5} \cline{1-1} 
132 &&  $\displaystyle \left\lceil \frac{(n-2)^2}{2} \right\rceil + 1$&Theorem~\ref{thm:132-321} & A061925\\ \cline{3-5} \cline{1-1} 
213 &&  $\displaystyle \binom{n-2}{2} + \binom{\lceil\frac{n-2}{2}\rceil}{2}+ 2$ & Theorem~\ref{thm:213-321}& A085787+2\\ \cline{3-5} \cline{1-1} 
231 && eventually 0 &Theorem~\ref{thm:231-321} & A000004\\ \cline{3-5} \cline{1-1} 
312 && $\displaystyle \sum_{k=0}^{\lfloor \frac{n-2}{2} \rfloor}\  \sum_{j=0}^{n-2-2k} {n-2-k-j \choose k}{2k \choose j}$& Theorem~\ref{thm:312-321}& A129847\\ \cline{3-5} \cline{1-1} 
321 && $ \displaystyle 1 + 2\binom{\lceil\frac{n+1}{2}\rceil}{3} +\begin{cases} 0 & \text{if } n \text{ is even},\\ 
 { \frac{n+1}{2} \choose 2} & \text{if } n \text{ is odd}. \end{cases}$ & Theorem~\ref{thm:321-321}& $\begin{cases} \text{A064999} & \text{if } n \text{ is even,} \\ \text{A056520} & \text{if } n \text{ is odd.} \end{cases}$\\ \hline
\end{tabular}
\caption{Included in the table are the cases where $\tau \in\{213,231,321\}$. The cases where $\tau\in\{123,132\}$ are trivial and are summarized in Theorem~\ref{thm:123 and 132}. The case where $\tau=312$ is similar (by symmetry) to the case where $\tau=231$ and is summarized by Theorem~\ref{thm:312}. }
\label{figtable}
\end{figure}

\subsection{Definitions and Notation}

Let $[n]$ denote the set $\{1,2,\ldots, n\}$, and for nonnegative integers $n_1 \leq n_2$, let $[n_1,n_2]$ denote the set of consecutive integers $\{n_1, n_1+1, \ldots, n_2-1,n_2\}$; if $n_1 > n_2$, then $[n_1,n_2]$ is the empty set.

We denote by $\S_n$ the set of permutations on $[n]$ and denote by $\C_n$ the set of cyclic permutations in $\S_n$, i.e~those permutations composed of exactly one cycle. For any $\pi\in\S_n$, the one-line notation of $\pi$ is $\pi = \pi_1\pi_2\cdots\pi_n$ where $\pi_i := \pi(i)$. The standard cyclic notation of a cyclic permutation $\pi\in\C_n$, denoted $C(\pi)$, is denoted in this paper by $C(\pi) = (c_1, c_2, \ldots, c_{n})$ where $c_1=1$ and $c_i = \pi_{c_{i - 1}}$ for $2\leq i\leq n$. For example, the permutation $\pi = 46152837$ is a cyclic permutation in $\C_8$ with $C(\pi) = (14526873)$.

For any permutation $\pi=\pi_1\pi_2\ldots\pi_n\in\S_n$, we say that  $\pi$ avoids a pattern $\sigma\in\S_k$ if there is no $i_1<i_2<\cdots<i_k$ with $\pi_{i_1}\pi_{i_2}\ldots \pi_{i_k}$ in the same relative order as $\sigma_1\sigma_2\ldots\sigma_k$. For example, the permutation $\pi = 251683497$ avoids the pattern $321$ since there is no subsequence of $\pi$ of length 3 that is in decreasing order. 
We will also consider pattern avoidance within a cycle. If $\pi$ is a cyclic permutation with $C(\pi) = (c_1,c_2,\ldots,c_n)$, then we say $C(\pi)$ avoids a pattern $\sigma = \sigma_1\sigma_2\ldots \sigma_n$ if there is no $i_1<i_2<\cdots<i_k$ with $c_{i_1},c_{i_2},\ldots,c_{i_k}$ in the same relative order as $\sigma$. 

Let us denote by $\A_n(\sigma;\tau)$ the set of cyclic permutations that avoid $\sigma$ in their one-line form and avoid $\tau$ in their cycle form, and let $a_n(\sigma;\tau)=|\A_n(\sigma;\tau)|$. For example, the permutation $\pi = 7 3 4 1 2 5 6= (1, 7, 6, 5, 2, 3, 4)$ avoids $\sigma = 132$ in the one-line notation since the word 7341256 contains no subsequence $\pi_{i_1}\pi_{i_2}\pi_{i_3}$ with $\pi_{i_1}<\pi_{i_3}<\pi_{i_2}$ and avoids the pattern $\tau=231$ in the cycle notation since $(1, 7, 6, 5, 2, 3, 4)$ contains no subsequence $c_{i_1}, c_{i_2}, c_{i_3}$ with $c_{i_3}<c_{i_1}<c_{i_2}$. Thus, we would say that $\pi\in\A_7(132;231)$.

\subsection{Insertion and deletion}
Within proofs throughout this paper, we will modify permutations by deleting or inserting elements into specific positions and shifting the remaining elements accordingly as seen in the definition and example below.

\begin{definition} Let $\pi = \pi_1\pi_2\cdots\pi_n$ be a permutation on the set $[n]$.
\begin{enumerate}
\item If $\pi'$ is formed by \emph{deleting} $\pi_i$ from $\pi$, then
\[ \pi' = \pi_1'\pi_2'\cdots\pi_{n-1}' \]
where
\[ \pi_j' = \begin{cases} \pi_j & \text{if } j < i \text{ and } \pi_j < \pi_i\\ \pi_j -1 & \text{if } j < i \text{ and }\pi_j > \pi_i\\
\pi_{j+1} &  \text{if } j \geq i \text{ and } \pi_{j+1} < \pi_i\\ \pi_{j+1} -1 &  \text{if } j \geq i \text{ and } \pi_{j+1} > \pi_i \end{cases}\] 
\item If $\pi'$ is formed by \emph{inserting} the element $k \in [1,n+1]$ into $\pi$ in position $i$, then
\[ \pi' = \pi_1'\pi_2'\cdots\pi_{n+1}' \]
where
\[ \pi_j' = \begin{cases} \pi_j & \text{if } j < i \text{ and } \pi_j < k\\ \pi_j +1 & \text{if } j < i \text{ and }\pi_j \geq k\\ k & \text{if } j=i\\
\pi_{j-1} &  \text{if } j \geq i \text{ and } \pi_{j} < k\\ \pi_{j-1} +1 &  \text{if } j \geq i \text{ and } \pi_{j} \geq k. \end{cases}\] 

\end{enumerate}
\end{definition}

For example, suppose $\pi = 37561248$. If $\pi'$ is formed from $\pi$ by inserting 3 in position 4, we would get $\pi' = 486371259$. If $\pi''$ is obtained by deleting 4 from $\pi$, we would get $\pi'' = 3645127$.

\subsection{Symmetries}
There are a few interesting symmetries on the set of permutations $\S_n$ that we will make use of in this paper. 

\begin{definition}
Given a permutation $\pi\in\S_n$, we define:
\begin{itemize}
\item the \textbf{reverse} of $\pi$, denoted $\pi^r$, by taking $\pi^r_i=\pi_{n-i+1}$ for each $i\in[n],$
\item the \textbf{complement} of $\pi$, denoted $\pi^c$, by taking $\pi^c_i=n+1-\pi_i$ for each $i\in[n],$
\item the \textbf{inverse} of $\pi$, denoted $\pi^{-1}$, by taking $\pi^{-1}_i=j$ where $\pi_j=i$,
\item the \textbf{reverse-complement} of  $\pi$, denoted $\pi^{rc}$, by taking $\pi^{rc}=(\pi^{r})^c$,
\item the \textbf{reverse-complement-inverse}  of $\pi$, denoted $\pi^{rci}$, by taking $\pi^{rci}=((\pi^{r})^c)^{-1}$,
\end{itemize}
For a cyclic permutation, we can apply the reverse, complement, or reverse-complement to the cycle notation $C(\pi)$ to get a new cyclic permutation.
\end{definition}

\begin{example}
If we consider $\pi = 4763125\in\S_7$, then: $\pi^{-1} = 5641732$, $\pi^r = 5213674$, $\pi^c = 4125763$, $\pi^{rc} = 3675214$, and $\pi^{rci} = 6517423$. Notice that $\pi$ is cyclic and $C(\pi) = (1,4,3,6,2,7,5)$. We can also consider $C(\pi)^r = (5,7,2,6,3,4,1)$ which in one-line notation is $5641732$. It is not a coincidence that this coincides with $\pi^{-1}$ as stated below in Lemma~\ref{lem:pi-and-Cpi-rci}. Similarly, we can write $C(\pi)^{c} = (7,4,5,2,6,1,3)$ which has one-line form $3675214$ and is equal to $\pi^{rc}$, and we can write $C(\pi)^{rc}  = (3,1,6,2,5,4,7)$ which has one-line form $6517423$ and is equal to $\pi^{rci}$. 
\end{example}

\begin{lemma}\label{lem:pi-and-Cpi-rci}
For any $n\geq 1$ and any cyclic $\pi\in\C_n$, we have:
\begin{itemize}
\item  $C(\pi^{-1}) = C(\pi)^r$ (up to cyclic rotation),
\item $C(\pi^{rc}) = C(\pi)^c$ (up to cyclic rotation),
\item $C(\pi^{rci}) = C(\pi)^{rc}$ (up to cyclic rotation).
\end{itemize}
\end{lemma}

\begin{proof}
Let $\pi = \pi_1\pi_2\cdots\pi_n$ with $C(\pi) = (1,c_2,c_3,\ldots, c_n)$  with $c_i=\pi_{c_{i-1}}$ for $i \in [2,n]$.  Let $k \in [1,n]$. We will consider the element after $k$ in the cycle structure of each permutation.

First consider $\pi^{-1}$ and choose $j$ so $\pi^{-1}_k = j$ where $\pi_j = k$. Thus, the element after $k$ in $C(\pi^{-1})$ is $j$. Notice $C(\pi)^r = (1,c_n, c_{n-1}, \ldots, c_2)$. Choose $i$ so $k= c_i$. The element after $k$ is then $c_{i-1}$. Since $c_i=\pi_{c_{i-1}}$ and $k= \pi_j$, the equation $k=c_i$ is equivalent to $\pi_j = \pi_{c_{i-1}}$. Thus the subscripts are equal and $c_{i-1}=j$ and $k$ maps to $j$ in $C(\pi)^r$ as desired.

Now consider $\pi^{rc} = (n+1-\pi_n)(n+2-\pi_{n-1})\cdots(n+1-\pi_1)$. Notice $k$ maps to $n+1-\pi_{n+1-k}$ in $\pi^{rc}$ so the element after $k$ in $C(\pi^{rc})$ is $n+1-\pi_{n+1-k}$. Now $C(\pi)^c = (n, n+1-c_2, n+1-c_3, \ldots, n+1-c_n)$. Choose $i$ so $k = n+1-c_i$. The element after $k$ in $C(\pi)^c$ is then $n+1-c_{i+1}$. Using the relationship $c_{i+1} = \pi_{c_i}$ along with $k=n+1-c_i$, we have the element after $k$  in $C(\pi)^c$ is $n+1-\pi_{c_i} = n+1-\pi_{n+1-k}$ as desired.

Finally, consider $\pi^{rci}$. By the previous two results, we have $C(\pi^{rci}) = C(\pi^{rc})^r = (C(\pi)^c)^r$. Because the complement and reverse operators commute, we have the desired result.
\end{proof}

\section{Enumerating $\A_n(\sigma; 123)$ and $\A_n(\sigma; 132)$}
In this section, we enumerate $\A_n(\sigma; 123)$ and $\A_n(\sigma; 132)$ for all patterns $\sigma$ of length 3.  The results are trivial as there is only one cyclic permutation whose cycle form avoids $123$, namely $\pi = n123\cdots(n-1)$, and only one cyclic permutation whose cycle form avoids $132$, namely $\pi = 234\cdots n1$. We state the results for completeness in the next theorems.

\begin{theorem} \label{thm:123 and 132} Suppose $n \geq 4$. Then: 
\begin{itemize}
\item $a_n(\sigma; 123) = \begin{cases} 0 & \text{for }  \sigma \in \{123, 312\}\\ 1 & \text{for } \sigma \in \{132, 213, 231, 321\},\end{cases}$
\item $a_n(\sigma; 132) = \begin{cases} 0 & \text{for }  \sigma \in \{123, 231\}\\ 1 & \text{for } \sigma \in \{132, 213, 312, 321\}.\end{cases}$
\end{itemize}
\end{theorem}

\section{Enumerating $\A_n(\sigma; 213)$}\label{sec:213}
In this section, we enumerate $\A_n(\sigma; 213)$ for all patterns $\sigma$ of length 3. The first lemma in this section, Lemma~\ref{lem:2n}, gives more information about permutations whose cycle avoids 213. The remainder of the section is divided into subsections examining each of the cases for $\sigma\in\S_3$ in more detail.

\begin{lemma}\label{lem:2n}
For $n\geq 6,$ if $\pi\in\A_n(\sigma;213)$ with $\sigma\in\{123,132,213,231,312\}$, then $\pi_1 \in \{2, n\}$.
\end{lemma}

\begin{proof} 
\sloppypar{Let $\pi\in\A_n(\sigma;213)$ with $C(\pi) = (1, c_2, c_3, \ldots, c_n)$, and suppose toward a contradiction that $\pi_1\in[3,n-1]$. Since $\pi_1 \neq 2$,  there is some $k\in[3,n]$ with $c_k=2$. Furthermore, since $C(\pi)$ avoids the pattern 213, $c_i>c_j$ for all $i \in [2, k-1]$ and $j \in [k+1,n]$. In particular, $n=c_i$ for some $i \in [3, k-1]$.} Therefore, $C(\pi)$ can be written as 
\[ C(\pi) = (1, c_2, \ldots, c_{i-1}, n, c_{i+1}, \ldots, c_{k-1}, 2, c_{k+1}, \ldots, c_n).\]

Let us achieve our contradiction by finding the pattern $\sigma$ for each possible choice of $\sigma$ in the one-line notation of $\pi$.  First, to see that there is a 132 pattern, we note that $1,n,$ and $\pi_n$ occur in positions $c_n$, $c_{i-1}$, and $n$, respectively. Since $c_n < c_{i-1} < n$, the pattern $1n\pi_n$ is a 132 pattern. To find a a 213 pattern, consider the elements $c_2$, $c_{k+1}$, and $n$, which occur in positions $1, 2,$ and $c_{i-1}$, respectively. Since $c_{k+1} < c_{2}$, the pattern $c_2c_{k+1}n$ is a 213 pattern. For a 312 pattern, we note the elements $c_2, 1$, and $2$ occur in positions $1, c_n$, and $c_{k-1}$, respectively. Since $c_n < c_{k-1}$, the pattern $c_212$ is a 312 pattern.

To see that $\pi$ has a 123 pattern, we consider two cases. If $i \neq k-1$, or equivalently $c_{k-1} \neq n$, then $12\pi_n$ is a 123 pattern occurring in positions $c_n, c_{k-1}$, and $n$. In the case where $c_{k-1} = n$, we note that $C(\pi) = (1, c_2, \ldots, c_{k-2}, n, 2, c_{k+1}, \ldots, c_n)$ where $k \neq 3$ since $\pi_1 \neq n$. Since $C(\pi)$ avoids 213,  $c_{k-2} = n-1$ otherwise $(n-1)c_{k-2}n$ would be a 213 pattern.  Thus $1(n-1)n$ is a 123 pattern in $\pi$ occurring in positions $c_n, c_{k-3},$ and $n-1$.

Finally, we will show that $\pi$ contains a 231 pattern. Because $C(\pi)$ avoids 213 and $\pi_1=c_2 \notin \{2,n\}$,  all elements greater than $c_2$ must appear in $C(\pi)$ before all elements smaller than 2.  Formally, the elements in $[c_2+1,n]$ must appear before the elements $[2,c_2-1]$ in $C(\pi)$. Thus, taking $r=n-c_2+2$, we have $c_r > c_2$ while $c_{r+1} < c_2$.  In one-line notation, the elements $c_2, c_3,$ and $c_{r+1}$ occur in positions $1, c_2,$ and $c_{r}$, respectively, and thus $c_2c_3c_{r+1}$ is a 231 pattern.

We have shown that when $\pi_1 \in [3,n-1]$, there is always a $\sigma$ pattern in the one-line notation for $\pi$ for all $\sigma \in  \{123,132,213,231,312\}$. Thus $\pi_1 \in \{2,n\}$ as desired. \end{proof}

\subsection{$\A_n(123; 213)$}\label{sec:123-213}
Permutations in $\A_n(123;213)$ are enumerated based on the value of $\pi_1$, and by Lemma~\ref{lem:2n}, $\pi_1 \in \{2,n\}$. The next lemma shows that there is only one cyclic permutation that avoids $123$ with $\pi_1=2$.  Because the proof does not depend on the condition that $C(\pi)$ avoids 213, we remove that condition from the hypotheses so that this lemma can also be used in later sections.

\begin{lemma}\label{lem:cyclic-123}
Suppose that $n\geq 2$ and that $\pi\in\C_n$ avoids the pattern 123.  If $\pi_1=2$, then
\[ C(\pi) = \left(1,2,n,3,n-1,4, \ldots, \left\lfloor \frac{n+3}{2}\right\rfloor \right). \]
\end{lemma}

\begin{proof}
Since $\pi_1=2$ and $\pi$ avoids 123, the remaining elements in $\pi$, excepting 1, must be decreasing.  Suppose $\pi_k = 1$ for some $k \in [2,n]$.   Thus $\pi_i = n+2-i$ for $i \in [2,k-1]$ and $\pi_i = n+3-i$ for $i \in [k+1,n]$. In cycle form, $C(\pi)$ then contains the cycle $(1, 2, n, 3, n-1, 4, \ldots, k)$. Since $\pi$ is cyclic, this cycle must have length $n$. Thus $k = \lfloor\frac{n+3}{2}\rfloor$.
\end{proof}

We now turn our attention to the permutations in $\A_n(123;213)$ where $\pi_1 = n$. In this case, we must also have $\pi_n=2$ as demonstrated in the lemma below.
\begin{lemma}\label{lem:123-213-4} Suppose $n \geq 3$ and $\pi \in \A_n(123;213)$ with $\pi_1=n$. Then $\pi_n=2$. \end{lemma}

\begin{proof} 
Suppose $\pi$ is a cyclic permutation with $\pi_1 =c_2=n$, and suppose toward a contradiction that $c_k=2$ for some $k \in [3,n-1]$. Then we can write
\[ C(\pi) = (1,n,c_3, c_4, \ldots, c_{k-1},2,c_{k+1}, c_{k+2}, \ldots, c_n).\]
Since $C(\pi)$ avoids 213, the elements in $C(\pi)$ before 2 must be larger than the elements in $C(\pi)$ after 2. In particular, $c_n < c_{k-1}$.  In one-line notation, 1 is in position $c_n$ and 2 is in position $c_{k-1}$, and thus $12\pi_n$ is a 123 pattern.  Therefore, we must have $\pi_n=2$. 
\end{proof}

Given a permutation $\pi \in \A_n(123;213)$ with the additional conditions that $\pi_1=n$ and $\pi_n=2$, we can simply delete $n$ and $2$ from $\pi$ to obtain a permutation $\pi' \in \A_{n-2}(123;213)$. This map is in fact a bijection, and can be generalized to other patterns as shown in the next lemma.
\begin{lemma} \label{lem:213-213-3}Suppose $n \geq 3$, $\sigma \in \{123,213\}$, and $\tau \in \{213, 231\}$. Then the number of permutations $\pi \in \A_n(\sigma;\tau)$ with $\pi_1=n$ and $\pi_n=2$ is equal to $a_{n-2}(\sigma; \tau)$.
\end{lemma}

\begin{proof} Let $\pi \in \A_n(\sigma; \tau)$ so that $\pi_1=n$ and $\pi_n=2$. Then $C(\pi) = (1,n,2,c_4, c_5, \ldots, c_{n})$. Let $\pi'$ be the permutation formed by deleting both $n$ and $2$ from $\pi$. Then $\pi'$ clearly still avoids $\sigma$. Also, $C(\pi') = (1,c_4-1,c_5-1 \ldots, c_{n}-1)$ and thus $\pi'$ remains cyclic and $C(\pi')$ avoids $\tau$. To see that we obtain every permutation in $\A_{n-2}(\sigma; \tau)$, we consider the process in reverse. Let $\pi' \in \A_{n-2}(\sigma;\tau)$. Form $\pi$ by inserting a 2 at the end followed by an $n$ in the front. This cannot create a new 123 or a 213 pattern and thus $\pi$ still avoids $\sigma.$ In cycle notation, this process is equivalent to inserting an $n$ followed by a 2 after the $1$ in $C(\pi')$. This insertion cannot create a new 213 or 231 pattern in $C(\pi)$, and thus the result follows.
\end{proof}

With these results in hand, we can now enumerate $\A_n(123;213)$.

\begin{theorem}\label{thm:123-213}
For $n\geq 4$, $a_n(123;213) = \lceil\frac{n}{2}\rceil + 1$.
\end{theorem}

\begin{proof} Let $n\geq 6$ and suppose $\pi \in \A_n(123;213)$. By Lemma~\ref{lem:2n}, either $\pi_1=2$ or $\pi_1 = n$. By Lemma~\ref{lem:cyclic-123}, there is exactly one permutation in $\A_n(123;213)$ with $\pi_1=2$. If $\pi_1=n$, then by Lemma~\ref{lem:123-213-4}, we must have $\pi_n=2$ and by Lemma~\ref{lem:213-213-3}, there are exactly $a_{n-2}(123;213)$ such permutations. Thus for $n \geq 6$, we have $a_n(123;213) = 1 + a_{n-2}(123;213)$. As $a_4(123;213)=3$ and $a_5(123;213)=4$, solving this recurrence yields the desired results.
\end{proof}

We note that of the total permutations in $\A_n(123;213)$, all but one of these permutations begin with $n$. Because this result will be referenced later in Section~\ref{sec:123-231}, we state the result as a corollary here.

\begin{corollary}\label{cor:123-213} For $n \geq 4$, the number of permutations $\pi \in \A_n(123;213)$ with $\pi_1 = n$ is $\lceil \frac{n}{2} \rceil$. \end{corollary}

\begin{example} Consider $\A_9(123;213)$. Lemma~\ref{lem:cyclic-123} yields the only permutation in $\A_9(123;213)$ with $\pi_1=2$. Namely $C(\pi) = (1,2,9,3,8,4,7,5,6)$ or $\pi = 298761543$.

The remainder of the permutations can be found recursively. We note that
\[ \A_7(123;213) = \{2765143, 7365142, 7541632, 7645132, 7651432\}. \] The proof of Lemma~\ref{lem:213-213-3} shows that we can get 5 new permutations in $\A_9(123;213)$ by inserting a 2 at the end and a 9 in the front of all of these permutations . Thus we have the following additional permutations in $\A_9(123;213)$:
\[ 938761542, 984761532, 986517432, 987451432, 987165432.\] In total, this gives us $\lceil\frac{9}{2}\rceil +1 = 6$ permutations in $\A_9(123;213)$.
\end{example}

\subsection{$\A_n(132;213)$}\label{sec:132-213}

Permutations in $\A_n(132;213)$ can be listed explicitly, and there are exactly $n-1$ permutations in this set.

\begin{theorem} \label{thm:132-213}
For $n\geq 2$, $a_n(132;213) =n-1$.
\end{theorem}

\begin{proof}
Let $n\geq 3$ and suppose $\pi \in \A_n(132;213)$.  By Lemma~\ref{lem:2n}, either $\pi_1=2$ or $\pi_1=n$, and we examine both of those cases now.

First, we claim that if $\pi \in \A_n(132;213)$ with $\pi_1=2$, then $\pi = 23\cdots n1$. 
Clearly, if $\pi$ avoids 132, we must have the elements in $[3,n]$ appear in increasing order, while 1 can appear anywhere. However, if $\pi_k=1$ for any $k$, then the cycle $(1,2,\ldots,k)$ will be part of the cycle decomposition of $\pi$. Since $\pi$ is cyclic, $k=n$.

Now suppose that $\pi \in \A_n(132;213)$ with $\pi_1=n$, and let $k \in [3,n]$ such $c_k = 2$. We claim that each $k$ induces a unique permutation thus providing exactly $n-2$ more permutations in $\A_n(132;213)$.  Since $C(\pi)$ avoids 213 and 
$C(\pi) = (1,n,c_3, c_4, \ldots, c_{k-1}, 2, c_{k+1,}, \ldots, c_n),$
we must have that all $c_i$'s appearing before 2 are greater than all elements appearing after 2. More formally, $\{c_3, c_4, \ldots, c_{k-1}\} = [n-k+3, n-1]$ and $\{c_{k+1}, c_{k+2}, \ldots, c_n\} = [3, n-k+2]$. Now in one-line notation, the element 1 is in position $c_n$ and the element 2 is in position $c_{k-1} > c_n$. Since $\pi$ avoids 132, there cannot be any elements between 1 and 2 in $\pi$, and thus $c_{k-1} = c_n+1$. This implies $c_{k-1} = n-k+3$ and $c_n = n-k+2$. Furthermore, since $c_n$ is the largest element in $\{c_{k+1},\ldots, c_n\}$ and $C(\pi)$ avoids 213, we must have that the elements after $2$ in $C(\pi)$ are increasing, and thus: 
\[ C(\pi) = (1,n,c_3, c_4,\ldots, c_{k-2}, n-k+3, 2, 3, 4, \ldots, n-k+2),\]
or equivalently,
$\pi = n34\cdots(n-k+2)12\pi_{n-k+4}\cdots\pi_n$. Since $\pi$ avoids 132, all elements in $\pi$ after 1 must be increasing and we have
\begin{equation}\label{eqn:132-213}
\pi = n34\cdots(n-k+2)12(n-k+3)(n-k+4)\cdots(n-1), \end{equation}
or equivalently,
\[C(\pi) = (1,n,n-1,\ldots, n-k+3, 2, 3, 4, \ldots, n-k+1, n-k+2),\]
Since $C(\pi)$ avoids 213, we see $\pi \in \A_n(132;213)$, and all $n-2$ possible choices for $k$ yield a unique permutation, our claim holds.
\end{proof}

We note that of the total permutations in $\A_n(132;213)$, all but one of these permutations begin with $n$. Because this result will be referenced later in Section~\ref{sec:213-231}, we state the result as a corollary here.

\begin{corollary}\label{cor:132-213} For $n \geq 3$, the number of permutations $\pi \in \A_n(132;213)$ with $\pi_1 = n$ is $n-2$. \end{corollary}

\begin{example} Consider $\A_7(132;213)$. Following the proof of Theorem~\ref{thm:132-213}, the only permutation in $\A_7(132;213)$ with $\pi_1=2$ is $2345671$. There are five permutations with $\pi_1=n$ based on the position of $2$ in the cycle notation.  These permutations are given by Equation~(\ref{eqn:132-213}) for $k \in [3,7]$ and thus:
\[ \A_7(132;213) = \{2345671, 7345612, 7345126, 7341256, 7312456, 7123456\}. \]
\end{example}

\subsection{$\A_n(213;213)$} \label{sec:213-213}
In this subsection, we consider the case where $\sigma=\tau=213$. If $\pi \in \A_n(213;213)$, the first lemma, Lemma~\ref{lem:213-213-1}, gives additional structural information about $\pi$.

\begin{lemma} \label{lem:213-213-1}Suppose $n \geq 6$ and $\pi \in \A_n(213;213)$.
\begin{itemize}
\item If $\pi_1=2$, then $\pi = 23\cdots n1$; and 
\item If $\pi_1=n$, then either $\pi_2=1$ or $\pi_n = 2$.
\end{itemize}
\end{lemma}

\begin{proof}
Let $\pi \in \A(213;312)$ and suppose first that $\pi_1=2$. Note that $\pi_n = 1$ because if $\pi_n \neq 1$, then $21\pi_n$ is a  213-pattern in one-line notation. Thus $C(\pi) = (1, 2,c_3, c_4, \ldots, c_{n-1},n).$  Suppose toward a contradiction that $\pi \neq 23\cdots n1$, or equivalently that $C(\pi) \neq (1,2, \ldots, n)$.  Then there is a smallest $k \in [3, n-2]$ where $c_k > k$, and $c_i=i$ for $i<k$. Note that $k\neq n-1$ because $c_{n-1} \leq n-1$. But then, $c_kkn$ is a 213 pattern in $C(\pi)$ since $\pi_{k-1}=c_k,$ $\pi_r=k$ for some $k+1\leq r\leq c_k$, and $\pi_s=n$ for some $s>c_k$. Therefore $c_i=i$ for all $i$ and $\pi=23\ldots n1$.

Now consider the case where $\pi_1=n$.  Thus $C(\pi) = (1, n, c_3, c_4, \ldots, c_{n})$. We want to show that either $c_n=2$ or $c_3=2$. Suppose toward a contradiction that $c_k=2$ for $k \in [4,n-1]$. Thus
\[ C(\pi) = (1,n,c_3, \ldots, c_{k-1}, 2, c_{k+1}, \ldots c_{n}) \] and
\[ \pi = n c_{k+1} \pi_3 \pi_4 \cdots \pi_{n-1} c_3.\]
Since $C(\pi)$ avoids 213, we must have $c_3 > c_{k+1}$. But then in one-line notation, $c_{k+1}2c_3$ is a 213 pattern which is a contradiction.  Thus either $c_3=2$ or $c_{n}=2$ which implies $\pi_n=2$ or $\pi_2=1$ as desired. \end{proof}

All permutations in $\A_n(213;213)$ that begin with $n1$ can be enumerated recursively as shown in the lemma below.

\begin{lemma} \label{lem:213-213-2}Suppose $n \geq 3$. Then the number of permutations $\pi \in \A_n(213;213)$ with $\pi_1=n$ and $\pi_2=1$ is equal to $a_{n-1}(213;213)-1$. 
\end{lemma}

\begin{proof} We will find a bijective correspondence between permutations in $\A_n(213;213)$ that begin with $n1$ and permutations in $\A_{n-1}(213;213)$ that begin with $n-1$. Let $\pi \in \A_n(213;213)$ so that $\pi_1=n$ and $\pi_2=1$. Then $C(\pi) = (1,n,c_3, c_4, \ldots, c_{n-1}, 2)$. Let $\pi'$ be the permutation formed by deleting $1$ from $\pi$. Then $\pi'$ clearly avoids 213. Also, $C(\pi') = (1,n-1,c_3-1, c_4-1, \ldots, c_{n-1}-1)$ and thus $\pi'$ remains cyclic and $C(\pi')$ avoids 213 as well. To see that we obtain every permutation in $\A_{n-1}(213;213)$ that begins with $n-1$, we consider the process in reverse.  Let $\pi' \in \A_{n-1}(213;213)$ with $\pi_1' = n-1$. Form $\pi$ by inserting a 1 in position 2 of $\pi'$ which is equivalent to inserting a 2 at the end of the cycle notation of $C(\pi')$. Because $\pi$ remains cyclic and avoids 213 in both one-line and cycle notation, $\pi \in \A_n(213;213)$. Thus our bijective correspondence holds.

By Lemma~\ref{lem:2n}, all permutations in $\A_{n-1}(213;213)$ either begin with $2$ or $n-1$, and Lemma~\ref{lem:213-213-1} states there is only one permutation in $\A_{n-1}(213;231)$ that does not begin with $n-1$. Thus there are $a_{n-1}(213;213)-1$ permutations $\pi \in \A_n(213;213)$ with $\pi_1=n$ and $\pi_2=1$.
\end{proof}

In the final case where $\pi \in \A_n(213;213)$ with $\pi_1=n$ and $\pi_n=2$, Lemma~\ref{lem:213-213-3} from Section~\ref{sec:123-213} states that we can can count the permutations recursively. Thus we are now ready to enumerate $\A_n(213;213)$.

\begin{theorem}\label{thm:213-213}
For $n\geq 1$, $a_n(213;213) = F_n$ where $F_n$ denotes the $n$-th Fibonacci number.
\end{theorem}

\begin{proof}
This is easily checked for $n \in [1,5]$, so let $n\geq 6$. We will count the number of permutations in $\A_n(213;213)$ by looking at three disjoint sets based on Lemma~\ref{lem:213-213-1}. First, we note that there is exactly one permutation $\pi \in \A_n(213;213)$ with $\pi_1 = 2$ by the first part of Lemma~\ref{lem:213-213-1}. In the second case, we consider the permutations $\pi \in \A_n(213;213)$ where $\pi_1=n$ and $\pi_2=1$. By Lemma~\ref{lem:213-213-2}, there are $a_{n-1}(213;213) -1$ such permutations.  For the third case, Lemma~\ref{lem:213-213-3} states that there are $a_{n-2}(213;213)$ permutations $\pi \in \A_n(213;213)$ with $\pi_1=n$ and $\pi_n=2$. Thus we have \[ a_n(213;213) = 1 + (a_{n-1}(213;213) -1)+ a_{n-2}(213;213) = a_{n-1}(213;213) + a_{n-2}(213;213)\] which satisfies the Fibonacci recurrence and the result follows.
\end{proof}

We note that of the total permutations in $\A_n(213;213)$, all but one of these permutations begin with $n$. Because this result will be referenced later in Section~\ref{sec:132-231}, we state the result as a corollary here.

\begin{corollary}\label{cor:213-213} For $n \geq 2$, the number of permutations $\pi \in \A_n(213;213)$ with $\pi_1 = n$ is $F_n - 1$ where $F_n$ denotes the $n$-th Fibonacci number. \end{corollary}

\begin{example} Consider $\A_7(213;213)$.  The first part of Lemma~\ref{lem:213-213-1} gives us the permutation 2345671 in $\A_7(213;213)$ as the only permutation beginning with 2. Next, we note that the permutations in $\A_6(213;213)$  that begin with 6 are $634512, 654132, 651342$, $614523,$ $615243,$ $612345,$ and $612534.$
The proof of Lemma~\ref{lem:213-213-2} shows that we can insert the element 1 into position 2 of all the permutations in this list to get permutations in $\A_7(213;213)$ that begin with 71. These seven permutations are:
\[ 7145623, 7165243, 7162453, 7125634, 7126354, 7123456, 7123645.\]
Finally, to find the remainder of the permutations in $\A_7(213;213)$, we list all five permutations in $\A_5(213;213)$ which are $23451$, $53412$, $54132$, $51234$, and $51423$.
The proof of Lemma~\ref{lem:213-213-3} shows that inserting a 2 at the end and a 7 at the beginning of these permutations will yield permutations in $\A_7(213;213)$ that begin with 7 and end in two; these five permutations are:
\[ 7345612, 7645132, 7651432, 7613452, 7615342.\]
Thus there are $F_7=13$ permutations in $\A_7(213;213)$.
\end{example}

\subsection{$\A_n(231;213)$}
This is the most complicated case for $\tau=213$. 
Permutations in $\A_n(231;213)$ can be enumerated via a recurrence based on the last element of the permutation. The lemmas below give additional constraints on $\pi_1$, and consider different cases for $\pi_n$. 

\begin{lemma} \label{lem:231-213-1} Suppose $n \geq 1$ and $\pi \in \A_n(231;213)$. Then $\pi_1=n$.
\end{lemma}

\begin{proof} 
We can easily check this for $n\in \{1,2\}$, so suppose $n \geq 3$ and let $\pi \in \A_n(231;213)$. By Lemma~\ref{lem:2n}, we know that $\pi_1\in\{2,n\}$. However, if $\pi_1 =2$, then to avoid 231 in the one-line form, we would have $\pi_2=1,$ which contradicts the fact that $\pi$ is cyclic. Thus $\pi_1 = n$.
\end{proof}

The next lemma shows that there is only one permutation in $\A_n(231;213)$ with $\pi_n=2$. Because the proof does not depend on the condition that $C(\pi)$ avoids 213, we remove that condition from the hypotheses so that the lemma can be used in future sections.

\begin{lemma} \label{lem:231-213-2} Suppose that $n\geq2$ and  that $\pi\in\C_n$ avoids 231. If $\pi_n=2$, then
\[ C(\pi) = \left(1, n, 2, n-1, 3, n-2, \ldots, \left\lfloor\frac{n+2}{2}\right\rfloor\right).\]
\end{lemma}

\begin{proof}
Suppose $\pi$ is as stated and $\pi_n=2$. Because $\pi$ avoids 213, the remaining elements in $\pi$, excepting 1, must be decreasing. Suppose $\pi_r =1$ for some $r \in [2,n-1]$. Then the remaining elements must be $\pi_i = n+1-i$ for $i \in [2,r-1]$ and $\pi_i = n+2-i$ for $i \in [r+1,n-1]$. In cycle form, $C(\pi)$ contains the cycle $(1,n,2,n-1,3,n-2,\ldots,r)$. Since $\pi$ is cyclic, this cycle must have length $n$ and thus $r=\lfloor\frac{n+2}{2}\rfloor$.
\end{proof}

In fact, the last lemma, Lemma~\ref{lem:231-213-2}, can be generalized to the case where $\pi_n=k$ for any $k\in[2,n-2]$ when $\pi \in \A_n(231;213)$. In this case, the first $n-k+3$ elements of $C(\pi)$ are forced.

\begin{lemma}\label{lem:231-213-3} Suppose $n \geq 4$ and $\pi \in \A_n(231;213)$ with $\pi_n=k$ for some $k \in [2,n-2]$. Then
\[ C(\pi) = \left(1,n,k,n-1,k+1,\ldots, \left\lfloor \frac{n+k}{2} \right \rfloor, k-1, c_{n-k+4}, c_{n-k+5}, \ldots, c_n\right).\]
\end{lemma}

\begin{proof} By Lemma~\ref{lem:231-213-1}, we have $\pi_1=n$. Since $C(\pi)$ avoids 231, we can write
\[ C(\pi) = (1,n,k,c_4, c_5, \ldots, c_{n-k+2}, c_{n-k+3}, \ldots, c_n)\]
where $\{c_4, \ldots, c_{n-k+2}\} = [k+1, n-1]$ and $\{c_{n-k+3}, \ldots, c_n\} = [2,k-1]$. In one-line notation, we have 
\[ \pi = n\pi_2\cdots\pi_{k-1}c_4\pi_{k+1}\cdots\pi_{n-1}k\]
where $\{\pi_2, \ldots, \pi_{k-1}\} = [1,k-1] \setminus \{c_{n-k+3}\}$ and $\{\pi_k, \ldots, \pi_{n-1}\} = [k+1,n-1] \cup \{c_{n-k+3}\}$.

We first claim that $c_4 = n-1$ and that $c_{n-k+3} = k-1$. Since $\pi$ avoids 231, all elements in $\pi$ that are greater than $k$ must appear in decreasing order.  In particular, since $\pi_k=c_4$, we have $c_4 \in \{n-1, c_{n-k+3}\}$. However, $k \neq n-1$, so $c_4 = n-1$. Also, if $c_{n-k+3} \neq k-1$, then $(k-1)c_4c_{n-k+3} = (k-1)(n-1)c_{n-k+3}$ is a 231 pattern in $\pi$. Thus $c_{n-k+3} = k-1$ as well.

Recall from the one-line notation of $\pi$ the elements in $\{\pi_k, \ldots, \pi_{n-1}\} = [k+1,n-1] \cup \{k-1\}$, excepting $k-1$, must appear in decreasing order.  Choose $r \in [k+1,n-1]$ so that $\pi_r = k-1$. Then $C(\pi)$ contains the cycle $(1,n,k,n-1,k+1, \ldots, r,k-1, c_{n-k+4}, \ldots, c_n)$. Since $\pi$ is cyclic, we must have $r=\lfloor \frac{n+k}{2} \rfloor$ as desired.
\end{proof}

Because the value of $\pi_n$ forces other parts of the permutation, we now recursively enumerate permutations in $\A_n(231;213)$ based on specific values of $\pi_n$.

\begin{lemma}\label{lem:231-213-7} Suppose $n \geq 4$. There are
\begin{itemize}
\item $a_{n-1}(231;213)$ permutations in $\A_n(231;213)$ with $\pi_n=n-1$;
\item $a_{n-3}(231;213)$ permutations in $\A_n(231;213)$ with $\pi_n=n-2$; and
\item $a_{n-1}(231;213)-a_{n-2}(231;213)$ permutations in $\A_n(231;213)$ with $\pi_n \in [2,n-3]$.
\end{itemize}
\end{lemma}

\begin{proof} 
Let $\pi \in \A_n(231;213)$, and suppose first that $\pi_n=n-1$. Then $C(\pi) = (1,n,n-1,c_4,\ldots,c_n)$ and $\pi = n\pi_2\ldots \pi_{n-1}(n-1)$. By deleting $n-1$ from $\pi$, we get $\pi'$ so that $C(\pi') =(1,n-1,c_4,\ldots,c_n)$ and $\pi'= (n-1)\pi_2\pi_3\ldots\pi_{n-1}$. Clearly, this process is reversible since inserting an $n-1$ at the end of the one-line notation cannot introduce a 231 pattern and inserting it after the $n-1$ in $C(\pi')$ cannot introduce a 213 pattern.

Next suppose $\pi_n=n-2$. Then by Lemma~\ref{lem:231-213-3}, $C(\pi) = (1,n,n-2,n-1,n-3,c_6,\ldots,c_n)$ and $\pi = n\pi_2\ldots \pi_{n-3}(n-1)(n-3)(n-2)$. By deleting $n-1, n-2,$ and $n-3$, we get $\pi'$ so that $C(\pi') =(1,n-3,c_6,\ldots,c_n)$ and $\pi'= (n-3)\pi_2\pi_3\ldots\pi_{n-3}$. Again, this process is reversible since inserting an $(n-1)(n-3)(n-2)$ at the end of the one-line notation cannot introduce a 231 pattern and inserting $n-2,n-1,n-3,$ after the $n-3$ in $C(\pi')$ cannot introduce a 213 pattern.

Finally, suppose $\pi_n=k$ for some $k\in[2,n-3]$. Then by Lemma~\ref{lem:231-213-3}, $C(\pi) = (1,n,k,n-1,k+1,n-2,k+2,\ldots,\floor{\frac{n+k}{2}},k-1, c_{n-k+4},\ldots,c_n)$. By letting $C(\pi')$ be the permutation formed by deleting $\floor{\frac{n+k}{2}}$ from $C(\pi)$, we get $C(\pi')$ still satisfies the conditions of Lemma~\ref{lem:231-213-3}. Since $c_{n-k+2} = \floor{\frac{n+k}{2}}\pm1$, this is equivalent to deleting $\floor{\frac{n+k}{2}}$ from the one-line notation. In particular, we note that $\pi'$ does not end in $n-2$ because $\pi$ did not end in $n-1$; there are  $a_{n-1}(231;213)-a_{n-2}(231;213)$ such permutations. Again, we reverse this process by first starting with $\pi' \in \A_{n-1}(231;213)$ where $\pi'$ does not end with $n-2$.  In this case, insert $\floor{\frac{n+k}{2}}$ into position $n+2-k$ of $C(\pi')$. \end{proof}

\begin{theorem}\label{thm:231-213}
For $n\geq 4$, the number of permutations in $\A_n(231;213)$ satisfies the recurrence
\[
a_n(231;213) = 2a_{n-1}(231;213)-a_{n-2}(231;213)+a_{n-3}(231;213).
\]
with $a_1(231;213)=a_2(231;213)=a_3(231;213)=1$. In closed form,
\[ a_n(231;213) =\displaystyle \sum_{k=0}^{\lfloor\frac{n-1}{3}\rfloor}\binom{n-1-k}{2k},\]
for all $n \geq 1.$
\end{theorem}

\begin{proof}
We can easily check that $a_1(231;213) = a_2(231;213)=a_3(231;213)=1$. Suppose $n \geq 4$ and let $\pi \in \A_n(231;213)$. By Lemma~\ref{lem:231-213-1}, $\pi_1=c_2=n$. By Lemma~\ref{lem:231-213-7}, we have 
\[
a_n(231;213) = 2a_{n-1}(231;213)-a_{n-2}(231;213)+a_{n-3}(231;213).
\]
Since the formula in the statement of the theorem satisfies this recurrence and the initial conditions, the theorem holds. 
\end{proof}

\begin{example} Consider $\A_7(231;213)$. Because we build this set of permutations recursively, we begin by listing $\A_4(231;213)$ and $\A_6(231;213)$: 
\begin{align*}
\A_4(231;213) &= \{4312, 4123\}\\
\A_6(231;213) &= \{631245, 612345, 614235, 641325, 612534,615243, 654132\}.\\
\end{align*}
To create those permutations in $\A_7(231;213)$ ending in 6, we start with the permutations in $\A_6(231;213)$ and insert a 6 at the end yielding the following permutations in $\A_7(231;213)$:
\[ 7312456, 7123456, 7142356, 7413256, 7125346,7152436, 7541326.\]
To create those permutations in $\A_7(231;213)$ ending in 5, we start with the permutations in $\A_4(231;213)$ and insert a 645 at the end giving us the permutations
\[ 7312645, 7123645.\]
Finally, to create those permutations in $\A_7(231;213)$ ending in $k$ for $k < 6$, we need the permutations in $\A_6(231;213)$ that do not end in 5. There are $a_6(231;213) - a_5(231;213) = 7-4=3$ such permutations: $(1,6,4,5,3,2),$ $(1,6,3,5,4,2),$ and $(1,6,2,5,3,4)$.  For each of these, let $k$ be the element after 6 in cycle form, and insert $\lfloor \frac{7+k}{2}\rfloor$ into position $9-k$ of the cycle form.  For the permutation $(1,6,4,5,3,2)$, $k=2$, so we insert $5$ in position $5$ of the cycle form giving the permutation $(1,7,4,6,5,3,2)$ in $\A_7(231;213)$. For the permutation $(1,6,3,5,4,2)$, $k=3$, so we insert 5 in position $6$ of the cycle yielding the permutation $(1,7,3,6,4,5,2)$. Finally, for the permutation $(1,6,2,5,3,4)$, $k=2$, and we insert $4$ in position $7$ of the cycle form to get the permutation $(1,7,2,6,3,5,4)$. In one-line form, these three permutations are
\[ 7126354, 7165243, 7651432.\]
Thus there are $7 + 2 + (7-4) =12$ permutations in $\A_7(231;213)$. In closed form, notice ${6 \choose 0} + {5 \choose 2} + {4 \choose 4} = 12$ as well.
 \end{example}

\subsection{$\A_n(312;213)$}
When $\sigma=312$ and $\tau=213$, there is only one permutation in $\A_n(312;213)$, namely the permutation $\pi= 2 3 \ldots n 1$ so that $C(\pi) = (1,2,3,\ldots, n)$.

\begin{theorem}\label{thm:312-213}
For $n\geq 1$, $a_n(312;213) = 1$.
\end{theorem}
\begin{proof} 
Let $n\geq 3$ and $\pi\in \A_n(312;213)$. By Lemma~\ref{lem:2n}, either $\pi_1=n$ or $\pi_1=2$. If $\pi_1=n$, then $\pi=n\pi_2\pi_3\cdots\pi_n$ must have $\pi_n=1$ in order to avoid 312. But this permutation is not cyclic for $n \geq 3$ since $\pi_1=n$ and $\pi_n=1$ yields a transposition. Thus $\pi_1=2$.

We claim that the permutation $\pi= 2 3 \ldots n 1$ is the only permutation in $\A_n(312;213)$.  To this end, let $C(\pi) = (1, 2, c_3, \ldots, c_n)$ and suppose toward a contradiction that there exists a $k$ so that $c_i=i$ for $i<k$ and $c_k >k$. Thus $c_\ell = k$ for some $\ell > k$ and
\[ C(\pi) = (1, 2, 3, \ldots, k-1, c_k, c_{k+1}, \ldots, c_{\ell-1}, k, c_{\ell+1}, \ldots, c_n).\]
If $\ell = n$, or equivalently $c_n=k$, then \[ \pi = 23\cdots(k-1)c_k1\pi_{k+1}\pi_{k+2} \cdots \pi_{n}.\] Since $c_k > k$, we have $c_k1k$ is a 312 pattern which is a contradiction. Thus $\ell \neq n$ and it must be that $c_{\ell-1} > c_n$ since $C(\pi)$ avoids 213. But then the pattern $c_k1k$, occurring in positions $k-1, c_n,$ and $c_{\ell-1}$, is a 312 pattern.  
\end{proof}

\subsection{$\A_n(321;213)$}
In this section, we enumerate the permutations in $\A_n(321;213)$. We begin by giving conditions on the position of 2 in cycle notation.
\begin{lemma}\label{lem:321-213} Suppose $n \geq 2$ and $\pi \in \A_n(321;213)$. Then either $\pi_1=2$ or $\pi_2=1$.
\end{lemma}

\begin{proof} If $n \in \{2,3\}$, the only permutations in $\A_n(321;213)$ are 21, 231, and 312 all of which satisfy the statement of the lemma.  So assume $n \geq 4$, and suppose toward a contradiction that $\pi_1 \neq 2$ and $\pi_2 \neq 1$. In $C(\pi)$, we thus have $c_k = 2$ for some $3 \leq k \leq n-1$ and for reference write
\[ C(\pi) = (1,c_2, \ldots, c_{k-1}, 2, c_{k+1}, \ldots, c_{n}).\]
Since $C(\pi)$ avoids 213, we must have $c_2 > c_{k+1}$. But then in one-line notation, $c_2c_{k+1}2$ is a 321 pattern which is a contradiction.
\end{proof}

For any permutation in $\A_n(321;213)$, we can insert either both a 2 at the beginning or a 1 in the second position to get a permutation in $\A_{n+1}(321;213)$. In fact, this process yields all possible permutations and we have the following theorem.

\begin{theorem}\label{thm:321-213}
For $n\geq 2$, $a_n(321;213) =2^{n-2}$.
\end{theorem}
\begin{proof}
We begin by showing that every permutation in $\A_{n}(321;213)$ corresponds to exactly two permutations in $\A_{n+1}(321;213)$.  Let $\pi \in \A_{n}(321;213)$, and let $\pi'$ be the permutation formed by inserting the element 2 in position 1 of $\pi$.  Thus $\pi' = 2\pi_1'\pi_2'\cdots\pi_n'$ where $\pi_i' = 1$ if $\pi_i=1$ and $\pi_i'=\pi_i + 1$ if $\pi_i \neq 1$. Notice that $\pi'$ avoids 321 in one-line notation. Furthermore, in cycle notation, this process is equivalent to inserting the element 2 after 1 and thus $C(\pi')$ avoids 213 as well. Thus $\pi' \in \A_{n+1}(321;213)$.

Similarly, let $\pi''$ be the permutation formed by inserting the element 1 in position 2 of $\pi$. Thus $\pi'' = \pi_1''1\pi_2''\cdots\pi_n''$ where $\pi_i'' = \pi_i + 1$ and thus $\pi''$ avoids 321.  Notice this process in equivalent to inserting 2 at the end of the cycle notation of $C(\pi)$ and thus $C(\pi'')$ also avoids 213, and $\pi'' \in \A_{n+1}(321;213)$. Furthermore, the only way for $\pi' = \pi''$ is if $\pi_1 = 1$ which is not true if $n \geq 2.$ Thus $a_{n+1}(321;213) \geq 2a_n(321;213)$.

To show equality, we note that this process is reversible.  If $\pi =\pi_1\pi_2\cdots\pi_{n+1} \in \A_{n+1}(321;213)$, then by Lemma~\ref{lem:321-213}, either $\pi_1 =2$ or $\pi_2 = 1$. In the case where $\pi_1=2$, deleting 2 from $\pi$ yields a permutation in $\A_n(321;213)$, and in the case where $\pi_2=1$, deleting 1 from $\pi$ yields a permutation in $\A_n(321;213)$.  Thus $a_{n+1}(321;213) = 2a_n(321;213)$. Since $a_2(321;213) = 1$, we have $a_n(321;213) =2^{n-2}$ as desired.
\end{proof}

\begin{example} Consider $\A_5(321;213)$.  To find the permutations in this set, we first list the permutations in $\A_4(321;213)$ and insert the element 2 in position 1 to each as well as inserting the element 1 in position 2.  The four permutations in $\A_4(321;213)$ are $2341,$ $3142,$ $2413,$ and $4123.$
Thus, the eight permutations in $\A_5(321;213)$ are
\[ \A_5(321;213) = \{23514, 25134, 24153, 23451, 31524, 51234, 41253, 31452\}.\]
\end{example}

\section{Enumerating $\A_n(\sigma; 231)$}\label{sec:231}

In this section we consider the cases where $\tau = 231$. Similar to the previous section, we begin with a lemma to give more information about permutations in $\A_n(\sigma;231)$ for certain patterns $\sigma$. The remainder of the section is divided into subsections examining each of the cases for $\sigma \in \S_3$ in more detail.

\begin{lemma}\label{lem:231}  
For any $n \geq 1$, if $\pi \in \A_n(\sigma; 231)$ with $\sigma \in \{132, 213, 231, 312\}$, then either $\pi_1=n$ or $\pi_n=1$.
\end{lemma}

\begin{proof} It is straightfoward to check this is true for $n<6$, so let us assume $n\geq 6$. Let $\pi \in \A_n(\sigma;213)$, and suppose toward a contradiction that $\pi_1 \neq n$ and $\pi_n \neq 1$. In cycle notation, this implies there is a $k \in [3,n-1]$ so that $c_k=n$.  Since $C(\pi)$ avoids 231, we must have $\{c_2, c_3, \ldots c_{k-1}\} = [2,k-1]$ and $\{c_{k+1}, c_{k+2}, \ldots, c_{n}\} = [k, n-1]$.

Let us achieve our contradiction by finding the pattern $\sigma \in  \{132, 213, 231, 312\}$ in the one-line notation of $\pi$. First, to see there is a 132 pattern, we note that $c_2$, $n$, and $c_{k+1}$ occur in positions $1, c_{k-1}$, and $n$, respectively, thus forming a 132 pattern. For the 213 pattern,  $c_2$ is in position 1, 1 is in position $c_n$ and $c_{k+1}$ is in position $n$. Since $c_2 < c_{k+1}$, this is a 213 pattern. Next, we have the pattern $c_2n1$ in $\pi$ which is a 231 pattern. And finally, $n1c_{k+1}$ is a 312 pattern. Therefore, for $\sigma \in \{132,213,231,312\}$, we have either $\pi_1=n$ or $\pi_n=1$ as desired.
\end{proof}

\subsection{$\A_n(123; 231)$}\label{sec:123-231}

We begin by considering the position of $n$ in the cycle notation of a permutation in $\A_n(123;231)$.

\begin{lemma}\label{lem:123-231-1} Suppose $n \geq 4$ and let $\pi \in \A_n(123;231)$ with $C(\pi) = (1,c_2, c_3, \ldots, c_n)$. If $c_k=n$ for $k \in [3,n-1]$, then \[ C(\pi) = \left(1, k-1, 2, k-2, 3, k-3, \ldots, \left\lfloor \frac{k+1}{2} \right\rfloor, n, k, n-1, k+1, \ldots, \left\lfloor \frac{n+k}{2}\right\rfloor\right).\]
\end{lemma}

\begin{proof} Since $C(\pi)$ avoids 231, elements in $C(\pi)$ that come before $n$ must be smaller than those that come after $n$.  Thus,
$\{c_2, c_3, \ldots, c_{k-1}\} = [2,k-1],$ and $\{c_{k+1}, c_{k+2}, \ldots, c_n\} = [k,n-1].$ In one-line notation, note that $\pi_n = c_{k+1}$. Since the elements in the set $[2,k-1]$ are all less than $c_{k+1}$ and $\pi$ avoids 123, these elements must appear in $\pi$ in decreasing order.  These elements occur in positions $[1, k-1]\setminus \{c_{k-1}\}$ of $\pi$, and thus we have
\[ \pi_i = \begin{cases} k-i & \text{if } i \in [1,c_{k-1}-1]\\ k+1-i & \text{if } i \in [c_{k-1}+1, k-1]. \end{cases} \]

On the other hand, we note that in one-line notation, $\pi_1 = c_2$, and thus all elements in $\pi$ greater than $c_2$ must be decreasing.  Since the elements in the set $[k,n-1]$ must occur in positions $[k,n] \setminus \{c_n\}$, we have
\[ \pi_i = \begin{cases} n+k-1-i & \text{if } i \in [k, c_n-1]\\ n+k-i & \text{if } i \in [c_n+1, n]. \end{cases}\]
Then $\pi$ contains the cycle $(1, k-1, 2, k-2, 3, \ldots, c_{k-1}, n, k, n-1, k+1, \ldots, c_n)$. Since $\pi$ is cyclic, $c_{k-1} = \lfloor (k+1)/2 \rfloor$ and $c_n = \lfloor (n+k)/2\rfloor$ as desired.
\end{proof}

The remaining cases occur when $n$ is either the second position in cycle notation or at the end. Lemma~\ref{lem:213-213-3} deals with the case where $c_2=n$ and $c_3=2$, and the following lemma examines the case where $c_2=n$ and $c_3 \neq 2$. In this case, we have a similar result to Theorem~\ref{thm:132-213}; the element $\pi_n=k$ uniquely determines the permutation if $k \in [4, n-2]$.

\begin{lemma}\label{lem:123-231-2} Suppose $n\geq 8$. Then there are $n-3$ permutations in $\pi \in \A_n(123;231)$ with $\pi_1=n$ and $\pi_n \neq 2$. \end{lemma}

\begin{proof}
Let $\pi \in \A_n(123;231)$ with $\pi_1=n$. We first show that $\pi_n \neq n-1$. If so, then the remaining elements of $\pi$ would need to be decreasing since $\pi$ avoids 123. Then $\pi_{n-1} = 1$ and $\pi$ contains the cycle $(1, n, n-1)$ which is a contradiction.

We now suppose $\pi_n=k$ for $k \in [4, n-2]$ and show that there is exactly one such permutation. Write
\[ C(\pi) = (1, n, k, c_4, c_5, \ldots, c_{k+1}, c_{k+2}, c_{k+3}, \ldots, c_n).\]
Since $C(\pi)$ avoids 231, we have $\{c_4, c_5, \ldots, c_{k+1}\} = [2,k-1]$ and $\{c_{k+2}, c_{k+3}, \ldots, c_n\} = [k+1,n-1]$. In one-line notation, note that $\pi_n=k$ so the elements in $[2,k-1]$, which occur in positions $[2,k] \setminus \{c_{k+1}\}$ must occur in decreasing order. In cycle notation, we then have $C(\pi) = (1, n, k, 2, k-1, 3, k-2, \ldots, c_{k+1}, c_{k+2}, \ldots, c_n)$. In order for $C(\pi)$ to contain all elements in $[1,k-1]$, we must have $c_{k-1} = \lfloor (k+2)/2 \rfloor$. Similarly, since $2$ comes before the elements in $[k+1, n-1]\setminus \{c_{k+2}\}$ in one-line notation, these elements must also be in decreasing order, and they appear in positions $[k+1,n-1] \setminus \{c_n\}$. Again, in order for $C(\pi)$ to contain all of the elements, we need $c_n = \lfloor (n+k)/2 \rfloor$. Thus, in cycle notation,
\begin{equation}\label{form-123-231-2}
C(\pi) =  \left(1, n, k, 2, k-1, 3, k-2, \ldots, \left \lfloor \frac{k+2}{2} \right\rfloor, n-1, k+1, n-2, k+2, \ldots, \left\lfloor \frac{n+k}{2} \right\rfloor\right). \end{equation}

Finally, we show that if $\pi_n=3$, there are exactly two possible permutations. Since $C(\pi)$ avoids 231, $c_4=2$ otherwise $3c_42$ would be a 231 pattern in $C(\pi)$.  For reference, note that in this case $C(\pi) = (1, n, 3, 2, c_5, c_6, \ldots, c_n)$.  We now claim that $c_5 \in \{4, n-1\}$ each producing a unique permutation with $n\geq 8$. Suppose toward a contradiction that $c_5 = k$ for $k \in [5, n-2]$. Since $C(\pi)$ avoids 231, we must have $\{c_6, c_7, \ldots, c_{k+1}\} = [4, k-1]$ and $\{c_{k+2}, c_{k+3}, \ldots, c_n\} = [k+1, n-1]$. In one-line notation, consider the pattern $2\pi_k\pi_{k+1}$. Since $\pi_k = c_6 < k$, we must have $\pi_{k+1} =1$ to avoid 123. But then $k=n-2$ since otherwise $2\pi_kc_{n-1}$ is a 123 pattern in $\pi$. If $k=n-2$, we have that $\pi_1\pi_2\pi_3 = n(n-2)2$ and $\pi_{n-1}\pi_n=13$. Since $\pi$ avoids 123, the remaining elements in $\pi$ must be decreasing and thus $\pi_4=n-1$ and $\pi_i=n+2-i$ for $i \in [4,n-2]$. But then $C(\pi)$ contains the cycle $(1,n,3,2,n-2,4,n-1)$ which is a contradiction. Thus $c_5 \notin [5,n-2]$.

Suppose that $c_5=4$. In this case $C(\pi) = (1,n,3,2,4,c_6, c_7, \ldots, c_n)$. In one-line notation, $\pi_1\pi_2\pi_3 = n42$ and $\pi_n=3$. The remaining elements in $\pi$, excluding 1, must be decreasing since $\pi$ avoids 123. Thus, $\pi_i = n+3-i$ for $i \in [4, c_n]$ and $\pi_i = n+4-i$ for $i \in [c_n+1, n-1]$. In cycle form, we have the cycle $(1, n, 3, 2, 4, n-1, 5, n-2, 6, \ldots, c_n)$. Since $\pi$ is cyclic, this cycle must contain all $n$ elements and thus $c_n = \lfloor (n+4)/2 \rfloor$ and a unique permutation in $\A_n(123;231)$ exists in this case, namely,
\begin{equation}\label{form-123-231-3}
C(\pi) = \left(1, n, 3, 2, 4, n-1, 5, n-2, 6, \ldots, \left\lfloor \frac{n+4}{2} \right\rfloor \right). \end{equation}
 A very similar argument shows that if $c_5=n-1$, the unique permutation in $\A_n(123;231)$ is 
\begin{equation}\label{form-123-231-4}
C(\pi) = \left(1, n, 3, 2, n-1, 4, n-2, 5, \ldots, \left\lfloor \frac{n+3}{2} \right\rfloor \right).\end{equation}
\end{proof}

The final lemma for this case counts the permutations in $\A_n(123;231)$ where $n$ is the last element in cycle form.  We enumerate these by making use of symmetries of permutations and the results in Section~\ref{sec:213}. The following lemma states a general result about enumerating permutations based on symmetries and will be used in future sections as well.

\begin{lemma}\label{lem:RC}  The number of permutations $\pi \in \A_n(\sigma; 231)$ with $\pi_n = 1$ is equal to the number of permutations $\pi' \in \A_n(\sigma^{rc}; 213)$ with $\pi'_1=n$. 
\end{lemma}

\begin{proof} Suppose $\pi \in \A_n(\sigma; 231)$ with $\pi_n=1$. Then $C(\pi) = (1,c_2, \ldots, c_{n-1}, n)$. Furthermore, we note that  $C(\pi)^c = (n, n+1-c_2, \ldots, n+1-c_{n-1},1)$ which we can cyclicly rotate to see $C(\pi)^c = (1,n, n+1-c_2, \ldots, n+1-c_{n-1})$. Notice that since $C(\pi)$ avoids 231, $C(\pi)^c$ avoids 213.

Let $\pi' = \pi^{rc}.$  We claim that $\pi' \in \A_n(\sigma^{rc}; 213)$ with $\pi'_1=n$. Since $\pi$ avoids $\sigma$, we see that $\pi^{rc}$ avoids $\sigma^{rc}$. Also, by Lemma~\ref{lem:pi-and-Cpi-rci}, we have $C(\pi') = C(\pi)^c$ which avoids 213. Thus $\pi' \in  \A_n(\sigma^{rc}; 213)$. Furthermore, this process works in reverse and thus equality holds.
\end{proof}

Since $123^{rc} = 123$, the number of permutations $\pi \in \A_n(123;231)$ where $\pi_n=1$ can thus be enumerated using Corollary~\ref{cor:123-213}. We can now enumerate all of $\A_n(123;231)$.

\begin{theorem} \label{thm:123-231}
For $n \geq 8$, the number of permutations in $\A_n(123;231)$ satisfies the recurrence
\[ a_n(123;231) = a_{n-2}(123;231) + 2n-6 + \left\lceil \frac{n}{2}\right\rceil \] where $a_6(123;231) = 12$ and $a_7(123;231)=20$.  In closed form,
\[ a_n(123;231) = 5{\lfloor\frac{n-1}{2}\rfloor \choose 2} +\begin{cases} \displaystyle{ 2n-5} & \text{if } n \text{ is even},\\ 
\displaystyle{n-2} & \text{if } n \text{ is odd}. \end{cases} \]
\end{theorem}

\begin{proof} It is not hard to verify that  $a_k(123;231)$ satisfies the results in the theorem for $k \in \{6,7\}$. Thus, we assume $n \geq 8$ and we enumerate permutations in $\A_n(123;231)$ based on the position of $n$ in cycle form. By Lemma~\ref{lem:123-231-1}, for each $k \in [3,n-1]$, there is exactly one permutation with $c_k=n$. Thus there are $n-3$ permutations in $\A_n(123;231)$ with $c_n=k$ for some $k \in [3,n-1]$.  By Lemma~\ref{lem:213-213-3} from Section~\ref{sec:213}, there are $a_{n-2}(123;231)$ permutations in $\A_n(123;231)$ with $c_2=n$ and $c_3=2$. In the case where $c_2=n$ and $c_3 \neq 2$, Lemma~\ref{lem:123-231-2} states that there are $n-3$ permutations.  Finally, Lemma~\ref{lem:RC} says the number of permutations in $A_n(123;231)$ with $c_n=n$ is equal to the number of permutations in $\A_n(123;213)$ where $1$ maps to $n$.  By Corollary~\ref{cor:123-213}, there are $\lceil \frac{n}{2}\rceil$ permutations with $c_n=n$. Thus we have
\[ a_n(123;231) = n-3 + a_{n-2}(123;231) + n-3 + \left\lceil \frac{n}{2} \right\rceil \]
and the desired recurrence relation is satisfied.
\end{proof}

\begin{example} Consider $\A_8(123;231)$. By Lemma~\ref{lem:123-231-1}, there are five permutations where $8$ is in position 3 through 7 in cycle notation, namely,
\[ (1, 2, 8, 3,7,4,6,5), (1,3,2,8,4,7,5,6), (1,4,2,3,8,5,7,6), (1,5,2,4,3,8,6,7), (1,6,2,5,3,4,8,7).\]
Using Lemma~\ref{lem:123-231-2}, we can list those permutations where 8 is in the second position of cycle notation and the third position is $k \in \{4,5,6\}$ according to Equation~(\ref{form-123-231-2}):
\[ (1,8,4,2,3,7,5,6), (1,8,5,2,4,3,7,6), (1,8,6,2,5,3,4,7).\]
Lemma~\ref{lem:123-231-2} also gives those permutations where 8 is in the second position and the third position is 3 in Equations~(\ref{form-123-231-3}) and (\ref{form-123-231-4}):
\[ (1,8,3,2,4,7,5,6), (1,8,3,2,7,4,6,5).\]

The permutations in $\A_8(123;231)$ where 8 is in the last position in cycle notation are the reverse complements of the permutations in $A_8(123;213)$ that begin with 8. There are 4 such permutations: $(1,8,2,7,3,6,4,5)$, $(1,8,2,7,3,5,6,4)$, $(1,8,2,7,3,4,6,5)$, and $(1,8,2,3,7,4,6,5)$. Taking the reverse complement of each of these (or equivalently taking the complement of the cycle and cyclicly shifting), we have the following corresponding permutations in $\A_8(123;231)$:
\[ (1,7,2,6,3,5,4,8), (1,7,2,6,4,3,5,8),  (1,7,2,6,5,3,4,8), (1,7,6,2,5,3,4,8).\]

Finally,  Lemma~\ref{lem:213-213-3} from Section~\ref{sec:213}, shows that there are $a_6(123;231) =12$ permutations in $\A_8(123;231)$ with $c_2=8$ and $c_3=2$. The twelve permutations in $\A_6(123;231)$ are 
\begin{align*} \A_6(123;231) = & \ \{(1,6,2,5,3,4),(1,6,3,2,5,4),(1,6,2,4,3,5),(1,6,3,2,4,5),(1,6,2,3,5,4),(1,6,4,2,3,5),\\
& \ \ (1,5,2,4,3,6),(1,5,4,2,3,6),(1,5,3,2,4,6),(1,4,2,3,6,5),(1,3,2,6,4,5),(1,2,6,3,5,4) \} \end{align*}
We obtain the additional twelve permutations in $\A_8(123;231)$ by inserting an 8 followed by a 2 after the 1 in cycle notation:
\begin{align*} &(1,8,2,7,3,6,4,5),(1,8,2,7,4,3,6,5),(1,8,2,7,3,5,4,6),(1,8,2,7,4,3,5,6),(1,8,2,7,3,4,6,5),\\
&(1,8,2,7,5,3,4,6),(1,8,2,6,3,5,4,7),(1,8,2,6,5,3,4,7),(1,8,2,6,4,3,5,7),(1,8,2,5,3,4,7,6),\\
&(1,8,2,4,3,7,5,6),(1,8,2,3,7,4,6,5). \end{align*}
\end{example}

\subsection{$\A_n(132; 231)$}\label{sec:132-231}

In this section we enumerate permutations in $\A_n(123;231)$. We begin by showing that for permutations in $\A_n(123;231)$, either 1 maps to $n$, or $n$ maps to 1. 

\begin{lemma}\label{lem:132-231-2} If $\pi\in\A_n(132;231)$, then either $\pi_1=n$ or $\pi_n=1$.
\end{lemma}
\begin{proof}
This is true for $n \in \{1,2,3\}$, so suppose $n\geq 4$ and $\pi\in\A_n(132;231)$. For the sake of contradiction, suppose $\pi_1\neq n$ and $\pi_n\neq 1$. Then there is some $k\in [3,n-1]$ with $c_k=n$ and we can write: \[C(\pi) = (1,c_2,\ldots,c_{k-1},n,c_{k+1},\ldots,c_n)\]
with each element in $\{c_2, \ldots, c_{k-1}\}$ less than each element in $\{c_{k+1},\ldots, c_n\}$. 
But now in one-line notation, we have that $c_2nc_{k+1}$ in positions $1, c_{k-1}$, and $n$, respectively, is a 132 pattern. Therefore, we must have $n\in\{c_2,c_n\}$, and the result follows.
\end{proof}

We proceed by considering the two possible positions of $n$ in cycle notation.  For the case where $\pi_n=1$, we can make use of permutation symmetries and use the results in Section~\ref{sec:213}. The following lemma considers the remaining case where $\pi_1=n$.

\begin{lemma}\label{lem:132-231-1}  For $n\geq3$, the number of permutations $\pi \in \A_n(132; 231)$ with $\pi_1=n$ is equal to $F_n-1$ where $F_n$ denotes the $n$-th Fibonacci number. 
\end{lemma}
\begin{proof}
Let $b_n$ denote the number of permutations $\pi \in \A_n(132; 231)$ with $\pi_1=n$. We have $b_3=1$ since the only such permutation is 312, and we have $b_4=2$ since the only permutations of length 4 satisfying these requirements are 4123 and 4312. Now, let $n\geq 5$. Since $F_n-1$ satisfies the recurrence:\[F_n-1=(F_{n-1}-1) + (F_{n-2}-1)+1,\]
it is enough to show that $b_n=b_{n-1}+b_{n-2}+1$. 

Let us first show that if $\pi_1=n$, then we must have either $\pi_n=n-1$ or $\pi_{n-1}=1$. By contradiction, let's suppose not and write
\[C(\pi) = (1,n,c_3,\ldots,c_{k-1},n-1,c_{k+1},\ldots,c_n)\]
for some $k\in[4,n-1]$ with each element in $\{c_3, \ldots, c_{k-1}\}$ less than each element in $\{c_{k+1},\ldots, c_n\}$. Then $1c_{k+1}c_3$ in positions $c_n, n-1$, and $n$, respectively, is an occurrence of 132 in the one-line notation of $\pi$. Therefore we must have $n-1\in\{c_3,c_{n}\}$. 

In the case where $n-1 = c_3$, we claim that the number of permutations with $\pi_1=n$ and $\pi_n=n-1$ is $b_{n-1}$. Indeed, if you have such a permutation, we can write $C(\pi) = (1, n, n-1, c_4, \ldots, c_n)$. We can obtain a permutation $\pi'\in\A_{n-1}(132;231)$ by deleting $n$ so that $C(\pi') = (1,n-1, c_4,\ldots,c_n)$. In the one-line notation, this only corresponds to deleting $n-1$ from the $n$-th position, and so no new patterns are created. On the other hand, starting with a permutation $\pi'\in\A_{n-1}(132;231)$ with $\pi'_{1}=n-1$, we can insert an $n-1$ into the $n$-th position of the one-line notation to obtain a permutation $\pi$, which is still a cyclic permutation since this corresponds to inserting $n-1$ after the $n$ in $C(\pi')$. Since $\pi'_1=n-1$, we must have $\pi_1=n$, and so  the $n-1$ in position $n$ cannot be part of a new 132 pattern.

Next, we show that if $\pi_1=n$ and $\pi_{n-1}=1$, we must have that either $\pi_n=n-2$ or $\pi_n=2$. For contradiction, suppose not. Then we have 
\[C(\pi) = (1,n,c_3,c_4\ldots,c_{k},c_{k+1},\ldots,c_{n-1},n-1)\]
where $k=c_3+1$, $\{c_4,\cdots,c_k\} = [2,c_3-1]$ and $\{c_{k+1}, \ldots, c_{n-1}\}=[c_3+1,n-2]$. Then $c_4(n-1)j$ in positions $c_3, c_{n-1}$, and $n$, respectively, is a 132 pattern. Thus $c_3\in\{2,n-2\}$. 

Notice that if $c_3=n-2$, then we have \[C(\pi) = (1,n,n-2,c_4,\ldots,c_{n-1},n-1)\] which means $\pi = n\pi_2\ldots \pi_{n-2}1(n-2)$. To avoid 132, we must therefore have $\pi_2=n-1$ since otherwise $\pi_2(n-1)(n-2)$ would be a 132 pattern. Thus $C(\pi) = (1,n,n-2,c_4,\ldots,c_{n-2},2,n-1).$ But now, since $C(\pi)$ must avoid 231, it must be the case that the remaining elements of $C(\pi)$, namely $\{c_4,c_5, \ldots, c_{n-2}\}$, must appear in decreasing order. Thus there is only one permutation with $\pi_1=n, \pi_{n-1}=1$, and $\pi_n=n-2$ given by
\begin{equation}\label{eqn:132-231-1}
C(\pi) = (1,n,n-2,n-3, n-4, \ldots, 3,2,n-1)
\end{equation} 

Finally, we consider those permutations with $\pi_1=n, \pi_{n-1}=1$, and $\pi_n=2$. In these cases, we have  
\[C(\pi) = (1,n,2,c_4,\ldots,c_{n-1},n-1), \] with $\pi = n\pi_2\ldots 12$. Notice that deleting $n$ and $2$ from both the cycle and the one-line notation leaves us with the permutation $\pi''\in \S_{n-2}$ with $C(\pi'') = (1, c_4-1, \ldots, c_{n-1}-1, n-2)$. Since this corresponds to deleting $n$ and $2$ from the one-line notation, we do not introduce any new patterns by this deletion. This is also reversible; if we start with a permutation $\pi''\in\A_{n-2}(132;231)$ with $\pi''_{n-2}=1$ and insert $n$ at the beginning and 2 at the end of $\pi''$, we will not introduce a new 132 pattern since 1 appeared at the end of $\pi''$. By Lemma~\ref{lem:RC}, there are $F_{n-2}-1$ permutations $\pi''\in\A_{n-2}(132;231)$ with $\pi''_{n-2}=1$, which inductively is $b_{n-2}$. 

Taken together, we have shown that $b_n=b_{n-1}+1+b_{n-2}$, from which it follows that $b_n=F_n-1$. 
\end{proof}

\begin{theorem} \label{thm:132-231}
For $n \geq 2$, $a_n(132;231) = 2F_{n} -2$ where $F_n$ denotes the $n$-th Fibonacci number. \end{theorem}

\begin{proof}
By Lemma~\ref{lem:132-231-2}, we have that $a_n(132;231)$ is the sum of the number of permutations in $\A_n(132;231)$ with $\pi_1=n$ and the number of permutations in $\A_n(132;231)$ with $\pi_n=1$.

Since $132^{rc}=213$, by Lemma~\ref{lem:RC}, the number of permutations $\pi \in \A_n(132;231)$ with $\pi_n=1$ is equal to the number of permutations $\pi' \in \A_n(213;213)$ with $\pi'_1=n$; Corollary~\ref{cor:213-213} states there are $F_n-1$ such permutations.  Finally, Lemma~\ref{lem:132-231-1} states there are  $F_n-1$ permutations with $\pi_1=n$, and thus the theorem follows.
\end{proof}

\begin{example} Consider $\A_6(132;231)$. We first list all permutations with $\pi_1=6$. Following the proof of Lemma~\ref{lem:132-231-1}, we begin by listing those permutations in $\A_6(132;231)$ that both start with 6 and end in 5. These are formed recursively by finding all permutations in $\A_5(132;231)$ that start with 5 and then insert a 5 at the end (which turns the original 5 into a 6).  The four permutations in $\A_5(132;231)$ that start with 5 are $54213, 53412, 51234,$ and $53124$. Inserting a 5 at the end yields the following permutations in $\A_6(132;231)$  that start with 6 and end in 5:
\[ 642135, 634125, 612345, 631245.\]
We then list the one permutation in $\A_6(132;231)$ that has $\pi_1=6$, $\pi_2=5$, and $\pi_5=1$ as given in Equation~(\ref{eqn:132-231-1}):
\[ (1,6,4,3,2,5).\]
For the remaining permutations that begin with 6, we start with permutations in $\A_4(132;231)$ that end in 1 and insert 6 in the front and 2 at the end. There are two permutations in $\A_4(132;231)$ that end in 1: $3421$ and $2341$. By inserting 6 in the front and 2 at the end, we have the following permutations in $\A_6(132;231)$:
\[ 645312, 634512.\] Notice there are a total of $F_5 - 1 = 7$ permutations in  $\A_6(132;231)$  that start with 6.

We now list the $F_5-1=7$ permutations in $\A_6(132;231)$ that end in 1. Lemma~\ref{lem:RC} states that these are found by first finding the permutations in $\A_6(213;213)$ that start with 6 and then taking the reverse-complement of each. The seven permutations in  $\A_6(213;213)$ that start with 6  are $634512$, $654132$, $651342$, $614523$, $615243$, $612345$, and $612534$. The reverse-complements of these permutations are:
\[ 562341, 546321, 534621, 452361, 435261, 234561, 342561.\]
Thus there are $2F_5 - 2 = 14$ total permutations in $\A_6(132;231)$.
\end{example}

\subsection{$\A_n(213; 231)$}\label{sec:213-231}

This section enumerates the set $\A_n(213;231)$. By Lemma~\ref{lem:231}, we know that for $\pi \in \A_n(213;231)$ either $\pi_1=n$ or $\pi_n=1$. For the case where $\pi_n=1$, we use symmetries of permutations and the results from Section~\ref{sec:132-213}. We enumerate the case where $\pi_1=n$ in the subsequent lemma.

\begin{lemma}\label{lem:213-231-2} Suppose $n \geq 2$. The number of permutations $\pi \in \A_n(213;231)$ with $\pi_1=n$ is equal to $a_{n-2}(213;231) + 1$. \end{lemma}

\begin{proof} Suppose $\pi \in \A_n(213;231)$ with $\pi_1=n$. In the case with the additional condition that $\pi_n=2$,  by Lemma~\ref{lem:213-213-3}, there are exactly $a_{n-2}(213;231)$ such permutations.

Assume then that $\pi_n \neq 2$. We claim that there is only one permutation satisfying these conditions, namely $\pi = n123\cdots(n-1)$, or equivalently, $C(\pi) = (1, n, n-1, n-2, \ldots, 3, 2)$. Suppose there is another permutation in $\A_n(213;231)$ with $\pi_1=n$ and $\pi_n \neq 2$. Choose $k\geq 2$ to be the first position in cycle notation with $c_k\neq n-k+2$.  That is, let $k \in [3,n-1]$ so that $c_i = n-i +2$ for $i \in [2,k-1]$ while $c_k < n-k+2$. For ease of notation, set $c_k=r$. Since $C(\pi)$ avoids 213, all elements smaller than $r$ must come before elements larger than $r$, and we have $\{c_{k+1}, c_{k+2}, \ldots, c_{k+r-2}\} = [2,r-1]$ while $\{c_{k+r-1}, c_{k+r}, \ldots, c_n\} = [r+1,n-k+2]$. In one-line notation, we have $2=\pi_i$ for some $i \in [3,r]$, 1 is in position $c_n \in [r+1,n-k+2]$, and $r$ is in position $n-k+3$. Thus $21r$ is a 213 pattern in one-line notation, and the only permutation in $\A_n(213;231)$ with $\pi_1=n$ and $\pi_n \neq 2$ is $\pi = n123\cdots(n-1)$.
\end{proof}

\begin{theorem} \label{thm:213-231}
For $n \geq 3$, the number of permutations in $\A_n(213;231)$ satisfies the recurrence
\[ a_n(213;231) = a_{n-2}(213;231)  + n -1 \]
where $a_1=a_2=1$. In closed form, $a_n(213;231) = \left\lfloor \frac{n^2}{4} \right\rfloor$.
\end{theorem}

\begin{proof} By Lemma~\ref{lem:231}, all permutations $\pi \in A_n(213;231)$ have either $\pi_1=n$ or $\pi_n=1$.  Lemma~\ref{lem:213-231-2} states that $a_{n-2}(213;231) + 1$ of these permutations have $\pi_1=n$.  Since $213^{rc} = 132$, by Lemma~\ref{lem:RC}, the number of permutations $\pi \in \A_n(213;231)$ with $\pi_n=1$ is equal to the number of permutations $\pi' \in \A_n(132;213)$ with $\pi'_1 = n$; Corollary~\ref{cor:132-213} states that there are $n-2$ such permutations. Thus the given recurrence relation holds.  The closed form can be found by solving the recurrence relation.
\end{proof}

\begin{example} Consider $\A_6(213;231)$. We first note that
$\A_4(213;231) = \{2341, 3421, 4312, 4123\}.$
The proof of Lemma~\ref{lem:213-213-3} shows that we can get four new permutations in $\A_6(213;231)$ by inserting a 2 at the end of each permutation and a 6 in the front.  Thus we have the following permutations in $\A_6(213;231)$:
\[ 634512, 645312, 654132, 651342.\]
The remaining permutation in $\A_6(213;231)$ with $\pi_1=6$ is found at the end of the proof of Lemma~\ref{lem:213-231-2} and is $612345.$

The permutations in $\A_6(213;231)$ with $\pi_6=1$ are found by taking the reverse-complements of those permutations in $\A_6(132;213)$ that start with 6. These permutations are $634512, 634125, 612345$, and $631245$.  Thus the remaining desired permutations in $\A_6(213;231)$ are
\[562341, 256341, 234561, \text{ and } 235641.\]
These are the 9 permutations in $\A_6(231;231).$
\end{example}

\subsection{$\A_n(231; 231)$}\label{sec:231-231}

We first show that for $\pi \in \A_n(231;231)$, $\pi_1=n$.
\begin{lemma}\label{lem:231-231-1}
Suppose $n \geq 1$ and  $\pi \in \A_n(231;231)$. Then $\pi_1=n$.
\end{lemma}

\begin{proof} 
If $n \in \{1,2\}$, the result clearly holds so assume $n \geq 3$. By Lemma~\ref{lem:231}, either $\pi_1=n$ or $\pi_n=1$. Suppose $\pi_n=1$. Since $\pi$ avoids 231, we must have the remaining elements in $\pi$ decreasing which implies $\pi_1=n$. But then $(1,n)$ is part of the cycle structure of $\pi$ which contradicts the fact that $\pi$ is cyclic. Therefore $\pi_1=n$ as desired.
\end{proof}

The permutations in $\A_n(231;231)$ are enumerated based on $\pi_n$.  The next lemma shows that we can recursively count those that end in $n-1$.

\begin{lemma}\label{lem:231-231-2} Suppose $n \geq 5$. The number of permutations $\pi \in \A_n(231;231)$ with $\pi_n = n-1$ is equal to $a_{n-1}(231;231)$.
\end{lemma}

\begin{proof} Suppose $\pi \in \A_n(231;231)$ with $\pi_n = n-1$. Using Lemma~\ref{lem:231-231-1}, we have $C(\pi) = (1,n,n-1, c_4, c_5, \ldots, c_n)$. Let $\pi'$ be the permutation formed by deleting $n-1$ from $\pi$ in one-line notation. Notice $C(\pi') = (1,n-1,c_4-1, c_5-1, \ldots, c_n-1)$ and thus $\pi'$ avoids 231 in both one-line notation and cycle notation. To see that we obtain every permutation in $\A_{n-1}(231; 231)$, we let $\pi' \in \A_{n-1}(231; 231)$ and form $\pi$ by inserting $n-1$ at the end of $\pi'$. In cycle notation, this is equivalent to inserting $n-1$ after $n$. This insertion cannot create a 231 pattern in either $\pi$ or $C(\pi)$, and the result follows.
\end{proof}

To count the remaining permutations in $\A_n$, we show that if $\pi \in A_n(231;231)$ and $\pi_n \neq n-1$, then $\pi_n =2$, and there is only one such permutation.  This result is part of the proof of the following main result.

\begin{theorem} \label{thm:231-231}
For $n \geq 6$, $a_n(231;231) = n$.
\end{theorem}

\begin{proof} This is easily checked for $n =6$, so assume $n \geq 7$. Suppose $\pi \in \A_n(231;231)$. By Lemma~\ref{lem:231-231-1}, $\pi_1=n$. If $\pi_n=n-1$, there are $a_{n-1}(231;231) $ such permutations so we consider the case where $\pi_n \neq n-1$ and show there is exactly one such permutation. By Lemma~\ref{lem:231-213-2}, there is exactly one permutation if $\pi_n=2$, and we show there are no permutations with $\pi_n \in [3,n-2]$.

Toward a contradiction, suppose $C(\pi) = (1,n,r,c_4, c_5, \ldots, c_n)$ where $r \in [3,n-2]$. Because $C(\pi)$ avoids 231, we have $\{c_4, c_5, \ldots, c_{r+1}\} = [2,r-1]$ and $\{c_{r+2}, c_{r+3}, \ldots, c_{n}\} = [r+1, n-1]$. If $\pi_{r+1} \neq 1$, then $\pi_r\pi_{r+1}1$ is a 231 pattern in one-line notation since $\pi_r \in[2,r-1]$. Thus assume $\pi_{r+1}=1$, or equivalently, $c_n = r+1$. Since $C(\pi)$ avoids 231, we must have the remaining elements in $[r+2,n-1]$ appearing in decreasing order and thus $c_i = n+r+1-i$ for $i \in [r+2, n]$. In one-line notation, if $r \leq n-4$, we have the pattern $(r+1)(r+2)r$ occurring in positions $r+2, r+3$, and $n$ which is a contradiction. Thus $r \in [n-3,n-2]$. Now in one-line notation, $\pi_{r+1} = 1$ so the elements before 1 must be decreasing.  We have $\{\pi_2, \pi_3, \ldots, \pi_r\} = [2,r-1]\cup\{n-1\}$, and so $\pi_2=n-1$ and $\pi_i=r+2-i$ for $i \in [3,r]$. In the case where $r=n-3$, in cycle form we have the cycle $(1,n,n-3,2,n-1,n-2)$, and in the case where $r = n-2$, we have the cycle $(1,n,n-3,2,n-1)$. In either case, $\pi$ is not cyclic since $n \geq 7$. Thus no such $\pi$ exists.

Therefore, if $\pi_n \neq n-1$, we must have $\pi_n=2$, in which case the permutation in Lemma~\ref{lem:231-213-2} is in $\A_n(231;231)$. Thus for $n \geq 7$, we have $a_n(231;231) = a_{n-1}(231;231) +1$. Solving the recurrence yields the desired result.
\end{proof}

\begin{example} Consider $\A_7(231;231)$. Because we build this set of permutations recursively, we first list $\A_6(231;231)$:
\[ \A_6(231;231) = \{631245, 612345, 641325, 642135, 652143, 654132\}.\]
To create those permutations in $\A_7$ ending in 6, we insert a 6 into the last position of all the permutations in $\A_6(231;231)$ to get the following permutations in $\A_7(231;231)$:
\[ 7312456, 7123456, 7413256, 7421356, 7521436, 7541326.\]
Finally, by Lemma~\ref{lem:231-213-2}, the only permutation in $\A_7(231;231)$ that has $\pi_n=2$ is 7651432. This permutation, along with the previous six permutations, make up all seven permutations in $\A_7(231;231)$.
\end{example}

\subsection{$\A_n(312; 231)$}

In this section, we make use of Lemma~\ref{lem:RC} together with some of the results in Section~\ref{sec:213}. Before doing so, we briefly explain why permutations in the set $\A_n(312;231)$ must end in 1.
\begin{lemma} \label{lem:312-231-1} Suppose $n \geq 1$ and $\pi \in \A_n(312;231)$. Then $\pi_n=1$.
\end{lemma}

\begin{proof} 
Lemma~\ref{lem:231} states that we must have $\pi_n=1$ or $\pi_1=n$. However if $\pi_1=n$, then since $\pi$ avoids 312, it is the decreasing permutation which is only cyclic when $n\leq 2$. Thus the lemma holds. 
\end{proof}

Because $\pi_n=1$ for all $\pi \in \A_n(312;231)$, Lemma~\ref{lem:RC} states that $a_n(312;231)$ is equal to the number of permutations $\pi'\in\A_n(231;213)$ that have $\pi'_1=n$.   Lemma~\ref{lem:231-213-1} states that all elements in $\A_n(231;213)$ have $\pi'_1=n$ and so $a_n(312;231) = a_n(231;213)$. We summarize these results in a theorem.

\begin{theorem}\label{thm:312-231}
For $n\geq 1$, $a_n(312;231) =\displaystyle \sum_{k=0}^{\lfloor\frac{n-1}{3}\rfloor}\binom{n-1-k}{2k}$.
\end{theorem}

\subsection{$\A_n(321; 231)$}
In this section, we show that all permutations in $\A_n(321;231)$ must be of the form
\[ 23\cdots (k-1) n 1 k(k+1)\cdots (n-1) = (1,2,\ldots, k-1, n, n-1, n-2, \ldots, k)\]
where $k \in [2, n]$.

\begin{theorem} \label{thm:321-231} Suppose $n \geq 2$. Then $a_n(321;231) = n-1$.
\end{theorem}

\begin{proof}  It is easy to check the result holds for $n \in [2,5]$, so suppose $n \geq 6$ and let $\pi \in \A_n(321;231)$. We will show that the position of $n$ in $C(\pi)$ uniquely determines the remainder of the permutation. To that end, choose $k \in [2,n]$ so that $c_k=n$. Because $C(\pi)$ avoids 231, we have $\{c_2, c_3, \ldots, c_{k-1}\} = [2,k-1]$ and $\{c_{k+1}, c_{k+2}, \ldots, c_n\} = [k,n-1]$. If $k=n$, then $\pi_n=1$, and so all the remaining elements in $\pi$ must be increasing since $\pi$ avoids 321. Thus $\pi=23\cdots n1$ which is in $\A_n(321;231)$.

If $k <n$, in one-line notation we have $n$ in position $c_{k-1}$, and 1 in position $c_n > c_{k-1}$. Because $\pi$ avoids 321, there can be no other elements between $n$ and $1$ and so $c_n = c_{k-1} +1$, which implies $c_n=k$ and $c_{k-1} = k-1$. Since $\pi$ avoids 321, the elements before $1$ must be increasing.  We know that $\pi_i \in [2,k-1]$ for $i \in [1,k-2]$, and thus $\pi_i = i+1$ for $i \in [1,k-2]$. Similarly, the elements in $\pi$ that come after $n$ must be increasing and thus $\pi_i = i-1$ for $i \in [k+1,n]$. Thus $\pi = 23\cdots (k-1) n 1 k(k+1)\cdots (n-1)$ which is in $\A_n(321;231)$.
\end{proof}

\begin{example} The 6 permutations in $\A_7(321;231)$ are
\[ \A_7(321;231) = \{7123456, 2713456, 2371456, 2347156, 2345716, 2345671\}.\]
\end{example}

\section{Enumerating $\A_n(\sigma; 312)$}\label{sec:312}

In this case, we can observe that $a_n(\sigma; 312) = a_n(\sigma^{-1}; 213)$. Indeed  $C(\pi)=(1,c_2, \ldots, c_n)$ avoids 312 if and only if $C(\pi)^r$ avoids 213 and by Lemma~\ref{lem:pi-and-Cpi-rci} is equal to $C(\pi^{-1})$ up to cyclic rotation. In particular, $C(\pi^{-1}) = (1, c_n, \ldots, c_2)$. Since this cyclic rotation just moves the 1 from the end of $C(\pi)^r$ to the front, it doesn't change the avoidance of 213. Therefore, we have the following theorem. 

\begin{theorem}\label{thm:312}
For $n\geq 5$, \[a_n(\sigma;312) = \begin{cases}
\lceil \frac{n}{2}\rceil + 1 & \text{ if } \sigma=123 \\
n-1& \text{ if } \sigma=132 \\
F_n & \text{ if } \sigma=213 \\
1 & \text{ if } \sigma=231 \\
 \displaystyle\sum_{k=0}^{\lfloor\frac{n-1}{3}\rfloor}\binom{n-1-k}{2k}& \text{ if } \sigma=312 \\
2^{n-2} & \text{ if } \sigma=321. \\
\end{cases}\]
\end{theorem} 

\section{Enumerating $\A_n(\sigma; 321)$}\label{sec:321}

\subsection{$\A_n(123;321)$}

The number of permutations in $\A_n(123;321)$ is zero for $n \geq 9$. We start with the following lemma which says that if you have a consecutive increasing run in $C(\pi)$, this implies you have an increasing pattern appearing in $\pi$. 

\begin{lemma}\label{lem:123-321}
For a permutation $\pi$ with $C(\pi) = (1, c_2, c_3,\ldots, c_n)$. If there is some $i$ with $c_i<c_{i+1}<c_{i+2}<c_{i+3}$, then there is a 123 pattern in the one-line form of $\pi$. 
\end{lemma}
\begin{proof}
Since $c_i<c_{i+1}<c_{i+2}<c_{i+3}$, $\pi_{c_i}$ appears before $\pi_{c_{i+1}}$ which appears before $\pi_{c_{i+2}}$. Also, $\pi_{c_i}\pi_{c_{i+1}}\pi_{c_{i+2}} = c_{i+1}c_{i+2}c_{i+3}$ is a 123 pattern. 
\end{proof} 

\begin{theorem} \label{thm:123-321} For $n \geq 9$, $a_n(123;321) = 0$.
\end{theorem}

\begin{proof} It is straightforward to check that $a_9(123;321)=0$, so suppose $n \geq 10$. We will show that $\A_n(123;321)$ is the empty set. 

\emph{Case 1:} Assume first that 3 appears before 2 in $C(\pi)$. Then $c_2=3$ since otherwise $c_2,3,2$ is a 321 pattern in $C(\pi)$. In one-line notation, $\pi_1=3$ and thus $\pi_2 \in \{1,n\}$ since otherwise $3\pi_2n$ is a 123 pattern in $\pi$. If $\pi_2=1$, then in $C(\pi)$, all elements between 3 and 2 must be increasing and thus $C(\pi) = (1,3,4,5,\ldots, n,2)$. This permutation does not avoid $123$ if $n\geq 6$.

In the case where $\pi_2=n$, we  claim that $\pi_3 \in \{2, n-1\}$ since otherwise $3\pi_3(n-1)$ is a 123 pattern in $\pi$ when $n \geq 6$. If $\pi_3=2$, then $C(\pi) = (1,3,2,n,4,5,\ldots, n-1)$ because all elements after $n$ in $C(\pi)$ must be decreasing. This permutation does not avoid $123$ when $n\geq7$. On the other hand, if $\pi_3 = n-1$, then 2 must appear immediately after $n-1$ in $C(\pi)$ since it avoids 321. Then $C(\pi) = (1,3,n-1,2,n,4,5,\ldots, n-2)$ because all elements after $n$ in $C(\pi)$ must be decreasing. Again, this does not avoid $\pi$ if $n\geq 7$.

\emph{Case 2:} Now assume that 2 appears before 3 in $C(\pi)$. 
We first consider the value of $\pi_1$ and then examine two subcases based on two possible values of $\pi_2$. If $\pi_1=2$, then $\pi_2=n$ otherwise $2\pi_2n$ is a 123 pattern for $n \geq 4$. Then $C(\pi) = (1,2,n,3,4,\ldots,n-1)$ because all elements after $n$ in $C(\pi)$ must be increasing. In this case, $\pi$ has a 123 pattern for $n \geq 6$. If $\pi_1=n$, then $C(\pi) = (1,n,2,3,\ldots,n-1)$ and $\pi$ has a 123 pattern if $n \geq 6$. So assume $n \geq 5$ and $\pi_1 = k$ for some $k \in [4,n-1]$.

Let $\pi_2 = r$ for some $r \in [4,n]-\{k\}$. 
If $r < k$, then $r=3$ otherwise $kr3$ is a 321 pattern in $C(\pi)$. On the other hand, if $r > k$, then $r=n$ otherwise $krn$ is a 123 pattern in $\pi$. We continue this case by examining these two cases in detail.

\emph{Subcase 2A:} Assume that $\pi_2 = 3$, and for reference, write $C(\pi) = (1,k, \ldots, 2,3, \ldots).$ Then $\pi_3 \in \{1,n\}$ since otherwise $3\pi_3n$ is a 123 pattern in $\pi$. If $\pi_3=1$, then $c_n=3$, and thus we must have $k=4$ since otherwise $k, k-1, 3$ would be a 321 in $C(\pi)$. But then  $C(\pi) = (1,4,5, \ldots, n,2,3)$ for which the one-line notation $\pi$ has a 123 pattern for $n \geq 6$. On the other hand, if $\pi_3=n$, notice that we must have $C(\pi) = (1, k, \ldots, 2, 3, n, \ldots, c_n)$. Since $C(\pi)$ avoids 321, the elements before 2 are increasing and the elements after $n$ are increasing. For $n\geq 10$, it must be that there is an increasing segment of length $4$ in $C(\pi)$ and thus by Lemma~\ref{lem:123-321}, $\pi$ does not avoid 123. 

\emph{Subcase 2B:} For the final case, assume that $\pi_2=n$. We must have $C(\pi) = (1,k,\ldots, 2,n, \ldots)$. As above, since the elements before $2$ are increasing and the elements after $n$ are increasing, there must be an increasing subsequence of $C(\pi)$ if $n\geq 9$. Therefore by Lemma~\ref{lem:123-321}, $\pi$ does not avoid 123. 
\end{proof}

\subsection{$\A_n(132;321)$}

In this section, we begin by establishing some facts about permutations $\pi\in\A_n(132;321)$. In particular, we show that $\pi$ must end with a consecutive increasing run $12\ldots (n-k+1)$ for some $k$.

\begin{lemma} \label{lem:132-321-1} Suppose $n\geq 2$ and let $\pi \in \A_n(132;321)$ with $C(\pi) = (1, c_2, c_3, \ldots, c_{n-1}, k)$. Then $k \geq \lceil \frac{n+1}{2} \rceil$, and $\pi_{k}\pi_{k+1}\cdots\pi_n = 12\cdots(n-k+1)$.
\end{lemma}

\begin{proof}
First we will show $k \geq \lceil \frac{n+1}{2} \rceil$. If $k \in \{n-1,n\}$, the result holds. Suppose $k \leq n-2$, and choose $r \in [1,n-1]$ so that $\pi_r = n$. Note that $r < k$ otherwise $1, n, \pi_n$ is a 132 pattern in $\pi$ occurring in positions $k, r$, and $n$.  Since $\pi$ avoids 132, we have that all elements appearing before $n$ in $\pi$ must be larger than all elements appearing after $n$. In particular, $\pi_1 = c_2 \geq n-r+1$. Also, in $C(\pi)$, $c_2 \leq k+1$ otherwise $c_2, k+1, k$ is a 321 pattern in $C(\pi)$. Since $n-r+1 \leq c_2 \leq k+1$, this implies $r \geq n-k$. This, together with the fact that $r < k$ implies that  $k \geq \lceil \frac{n+1}{2} \rceil$.

Next we will show that $\pi_{k+1}\pi_{k+2}\cdots\pi_{n} = 23\cdots(n-k+1)$. Since $\pi_k=1$ and $\pi$ avoids 132, it is clear that $\pi_{k+1}\pi_{k+2}\cdots\pi_{n}$ is increasing. If we can show that $\pi_n=n-k+1$, then we would be done. Suppose instead that $\pi_n=\ell>n-k+1$ and there is some $s$ and $t$ with $r<s<k$ and $1<t<n-k+1$ with $\pi_s=t$. Then in $C(\pi)$, $t$ follows $s$ and $\ell$ follows $n$, so either $n, \ell, j$ is a 321 pattern, or $c_2,m,r$ is. 
Thus $\pi_{k+1}\pi_{k+2}\cdots\pi_{n} = 23\cdots(n-k+1)$.

\end{proof} 

We now count the permutations in $\A_n(132;321)$ based on the position of $1$. The next lemma shows that the total number of permutations that end in 1 is equal to the number of permutations of one size smaller. We follow this up with a lemma counting the number of such permutations that do not end in 1.

\begin{lemma} \label{lem:132-321-2}
Suppose $n \geq 2$. Then the number of permutations $\pi \in \A_n(132;321)$ with $\pi_n=1$ is equal to $a_{n-1}(132;321)$.
\end{lemma}

\begin{proof}
This is true for $n =2$, so assume $n \geq 3$ and let $\pi \in \A_n(132,321)$ with $C(\pi) = (1, c_2, c_3, \ldots, k,n)$ for $k \in [2,n-1]$. We first show that $k \geq \lceil \frac{n}{2} \rceil$. In one-line notation, since $\pi_k=n$, we must have that all the elements preceding $n$ are larger than all elements that come after $n$ since $\pi$ avoids 132. In particular, $\pi_1=c_2 \geq n-k+1$. Then if $k < \lceil \frac{n}{2} \rceil$, in $C(\pi)$, we have the pattern $c_2, c_2-1, k$ which is a 321-pattern. Thus we must have $k \geq \lceil \frac{n}{2} \rceil$.

Next we will show that $\pi_{k+1}\pi_{k+2}\cdots\pi_{n-1} = 23\cdots(n-k)$. This is very similar to the result in the previous lemma. Let $i,j \in [k+1,n-1]$ with $i <j$. Then in $C(\pi)$, $i$ must appear before $j$ otherwise $j, i, k$ would be a 321-pattern in $C(\pi)$. Since $i$ comes before $j$ we also must have $\pi_i$ comes before $\pi_j$ and thus have the pattern $c_2, \pi_i, \pi_j$ in $C(\pi)$. Since $c_2>\pi_i$ and $c_2 > \pi_j$, it must be that $\pi_i < \pi_j$. Thus $\pi_{k+1}\pi_{k+2}\cdots\pi_{n-1} = 23\cdots(n-k)$.

We can now prove our main result.  Let $\pi'$ be the permutation formed by deleting $n$ from $C(\pi)$. This is equivalent to taking the one-line notation of $\pi$, deleting $n$, and moving 1 to position $k$. Notice $\pi'$ is cyclic and $C(\pi')$ avoids 321. Also, $\pi' = \pi_1\pi_2\ldots\pi_{k-1}1\pi_{k+1}\pi_{k+2}\cdots\pi_{n-1}$. Because all elements after 1 are increasing, $\pi'$ avoids 132 and thus $\pi' \in \A_n(132;321)$. To see that we obtain every permutation in $\A_{n-1}(132;321)$, we consider the process in reverse. Let $\pi' \in \A_{n-1}(132;321)$. Form $\pi$ by inserting $n$ to the end of $C(\pi')$. Thus $\pi$ is cyclic and $C(\pi)$ avoids 321 still.  In one-line notation, inserting $n$ at the end of $C(\pi')$ is equivalent to replacing $1$ with $n$ and inserting 1 at the end. By Lemma~\ref{lem:132-321-1}, since all elements after 1 in $\pi'$ are smaller than the elements that come before it, $\pi$ is still 132 avoiding.
\end{proof}

\begin{lemma}\label{lem:132-321-3}Suppose $n \geq 2$. Then the number of permutations $\pi \in \A_n(132;321)$ with $\pi_n\neq1$ is equal to $2\lceil \frac{n}{2} \rceil- 3$.
\end{lemma}

\begin{proof} Suppose $\pi \in \A_n(132;321)$ and $\pi_k=1$ for some $k \geq \lceil \frac{n+1}{2} \rceil$. We begin by showing that for each $ \lceil \frac{n+1}{2} \rceil \leq k < n-1$, there is exactly one permutation with $\pi_k =1$. By Lemma~\ref{lem:132-321-1}, since $\pi_k=1$, we have $\pi_i = i-k+1$ for $i \in [k, n]$. In  $C(\pi)$, elements in $[k+1,n]$ must appear in increasing order (since $c_n=k$). Thus, 
 we must have that $C(\pi)$ is of the form \[C(\pi) =(1, \ldots, (k+1),2,\ldots, (k+2),3,\ldots, (n-1),(n-k),\ldots, n,(n-k+1), \ldots, k).\]
Therefore,
for each $i \in [2,n-k]$, either $\pi_i = k+i$ or $\pi_i \in [n-k+2,k-1]$. We claim that we must have $\pi_i = k+i$ for $i \in [2,n-k]$.  If not, then there is some $r \in [2,n-k]$ where $\pi_r=x$ for some $x \in [n-k+2,k-1]$. But then $r+k-1, x, r+1$ is a 321-pattern in $C(\pi)$.

Consider the one-line notation of $\pi$.  Because $\pi_i = k+i$ for $i \in [2,n-k]$ and $\pi_i = i-k+1$ for $i \in [k,n]$, we must have $\{\pi_1\} \cup \{\pi_{n-k+1}, \pi_{n-k+2}, \ldots, \pi_{k-1}\} = [n-k+2, k+1]$. Since $\pi$ avoids 132, $\pi_1 = k+1$. Thus $C(\pi) = (1, k+1, 2, k+2, 3, \ldots, n, n-k+1, c_{2(n-k+1)}, \ldots, c_{n-1}, k)$. Since $C(\pi)$ avoids 321, the remaining elements after $n$ in $C(\pi)$ must be increasing and thus
\[ \pi = (k+1)(k+2)\cdots n (n-k+2)(n-k+3)\cdots k12\cdots(n-k+1). \]

Now suppose $\pi_{n-1} = 1$, which implies by Lemma~\ref{lem:132-321-1} that $\pi_{n} = 2$. We count these permutations by the position of $n$ in cycle notation.  To that end, suppose first that $c_2 = n$.  Because all elements after $n$ is $C(\pi)$ must be increasing, we have $C(\pi) = (1, n, 2, 3, \ldots, n-1)$ which is in $\A_n(132;321)$. Suppose next that $c_{n-2} = n$. All elements before $2$ in $C(\pi)$ must be increasing so we have $C(\pi) = (1, 3, 4, \ldots, n-2,n,2,n-1)$. This permutation is not in $\A_n(132;321)$ because $3(n-1)4$ is a 132 pattern in $\pi$.

Finally, suppose that $c_r = n$ for some $r \in [3, n-3]$.  We claim that the number of permutations in $\A_n(132;321)$ with $\pi_{n-1} = 1$ and $c_r=n$ for $r \in [3,n-3]$ is equal to the number of permutations in $\A_{n-2}(132;321)$ with 1 in the second to last position. Since $C(\pi)$ avoids 321, all elements after $n$ must be increasing and all elements before 2 must be increasing. In particular, either $c_{r-1}=n-2$ or $c_{n-1} = n-2$. If $c_{n-1} = n-2$, then $\pi_1n(n-1)$ is a 132-pattern in $\pi$ occurring in positions $1, c_{r-1}$, and $n-2$, and thus we must have $c_{r-1}=n-2$. 

Similarly, we can consider the position of $n-3$ in $C(\pi)$. If $c_{r-2} = n-3$, then $\pi_1(n-1)(n-2)$ is a 132-pattern in $\pi$ and so $c_{n-1} = n-3$. Form $\pi'$ by deleting $n-1$ and $n$ from $\pi$. In cycle notation, this is equivalent to deleting both $n$ and $n-1$ as well.  Thus $C(\pi') = (1, c_2', c_3', \ldots, c_{r-2}', n-2, 2, c_{r+1}', c_{r+2}', \ldots, c_{n-3}', n-3)$ where $c_i' = c_i$ for $i \in [2, r-2]$ and $c_i' = c_{i+1}$ for $i \in [3, n-3]$. Thus $\pi' \in \A_{n-2}(132;321)$ with 1 in the second to last position. Because this process is reversible, we have the desired equality. Recalling that there is exactly one permutation with $\pi_{n-1}=1$ and $\pi_n=2$, we have that the number of permutations in $\A_n(132;321)$ with $\pi_{n-1} = 1$ is one more than the number of permutations in $\A_{n-2}(132;321)$ with $\pi_{n-3} = 1$. For the base cases, we note that there is one permutation in $\A_3(132;321)$ with $\pi_{2} =1$ and one permutation in $\A_4(132;321)$ with $\pi_3=1$. Solving this recurrence yields a total of $\lceil \frac{n-2}{2} \rceil$ permutations in $\A_n(132;321)$ with $\pi_{n-1}=1$.

Overall, we have shown that there are $n-1-\lceil \frac{n+1}{2} \rceil$ permutations $\pi \in \A_n(132;321)$ with $\pi_k = 1$ for $k < n-1$ and $\lceil \frac{n-2}{2} \rceil$ permutations with $\pi_{n-1} = 1$. Thus there are
\[ n-1-\left\lceil \frac{n+1}{2} \right\rceil + \left\lceil \frac{n-2}{2} \right\rceil = 2\left\lceil \frac{n}{2} \right\rceil - 3\] permutations in $\A_n(132;321)$ that do not end in 1.

\begin{theorem} \label{thm:132-321} For $n \geq 2$, $a_n(132;321) = \lceil \frac{(n-2)^2}{2} \rceil + 1$.
\end{theorem}

\begin{proof} By Lemmas~\ref{lem:132-321-2} and \ref{lem:132-321-3}, we have $a_n(132;321) = a_{n-1}(132;321) + 2\lceil \frac{n}{2} \rceil- 3$. Since $a_2(132;321) = 1$, solving this recurrence relation yields the desired results.
\end{proof}

\begin{example} Consider $\A_7(132;321)$. We list those permutations that end in 1 by considering the nine permutations in $\A_6(132;321) = \{634512, 564123, 456231, 345621, 435612, 534621, 435261, 234561, 342561\}$. Following the proof of Lemma~\ref{lem:132-321-2}, for each of these nine permutations, we replace the element 1 with 7 and insert a 1 at the end yielding the following nine permutations in $\A_7(132;231)$:
\[ 6345721, 5647231, 4562371, 3456271, 4356721, 5346271, 4352671, 2345671, 3425671.\]
Following the proof of Lemma~\ref{lem:132-321-2}, we note that there is exactly one permutation with $\pi_k=1$ for each $4 \leq k < 6$ yielding the following two permutations:
\[5671234, 6745123.\]
Finally, there are $\lceil \frac{7-2}{2} \rceil = 3$ permutations with $\pi_6=1$. If $\pi_1=n$, there is exactly one such permutation, namely
\[ 7345612.\]
If $\pi_1 \neq 1$, we must first find the permutations in $\A_5(132;321)$ that have 1 in the second to last position.  These two permutations are $53412$  and $34512$. For each of these permutations we insert the elements 6 and 7 before the 1 to get the following additional permutations in $\A_7(132;321)$:
\[ 5346712, 3456712.\]
\end{example}

\end{proof}

\subsection{$\A_n(213;321)$}\label{sec:213-321}
Let us first count the number of permutations in $\A_n(213;321)$ with $\pi_n=1$. 
\begin{lemma} \label{lem:213-321-1}
Suppose $n \geq 3$. Then the number of permutations $\pi \in \A_n(213;321)$ with $\pi_n=1$ is equal to $\lceil \frac{(n-3)^2}{2} \rceil + 1$.
\end{lemma}
\begin{proof}
We first show that the number of permutations $\pi \in \A_n(213;321)$ with $\pi_n=1$ is equal to the number of permutations $\pi \in \A_n(132;321)$ with $\pi_n=1$. Indeed, by Lemma~\ref{lem:pi-and-Cpi-rci}, the reverse-complement of the cycle is equal to the reverse-complement-inverse of the one-line notation. By taking a cycle of the form $C(\pi) = (1,c_2, \ldots, c_{n-1},n)$ that avoids 321, we can see that $C(\pi)^{rc} = (1, n+1-c_{n-1}, \ldots, n+1,c_2, n)$ avoids $321^{rc}=321$ with $C(\pi)^{rc}= C(\pi^{rci})$. Furthermore, $\pi$ avoids 213 if and only if $\pi^{rci}$ avoids $213^{rci} = 132$. 

By Lemma \ref{lem:132-321-2}, the number of $\pi \in \A_n(132;321)$ with $\pi_n=1$ is equal to the number of permutations in $\pi \in \A_{n-1}(132;321)$, which by Theorem~\ref{thm:132-321}, is equal to $\lceil \frac{(n-3)^2}{2} \rceil + 1$.
\end{proof}

To count the number with $\pi_n\neq 1$, we proceed by giving more information about the structure of $C(\pi)$ based on the element at the end of $C(\pi)$.

\begin{lemma}\label{lem:213-321-2}
Suppose $n\geq 3$ and let $\pi \in \A_n(213;321)$ with $\pi_k = 1$ where $k\neq n$. Then
\begin{itemize}
\item $k\geq \lceil \frac{n+1}{2}\rceil$;
\item $C(\pi)$ contains the consecutive terms
\[ k+1,2,k+2,3,k+3,4,\ldots, n,n-k+1,\] or equivalently, $\pi_i = i-k+1$ for all $i\in[k+1,n]$ and $\pi_i = k+i$ for all $i\in[2,n-k]$; and
\item If $\pi_r = k$, then $r \in \{k-2,k-1,n\}$.
\end{itemize}
\end{lemma}

\begin{proof} 
First, for the sake of contradiction, suppose that $k<\frac{n+1}{2}$. Then since $\pi_k = 1$ and $\pi$ avoids 213, $\{\pi_{k+1}, \ldots, \pi_n\} = [2, n-k+1]$ with $n-k+1>k$. In particular there is some $s>k$ so that $\pi_s = k+1$. Now, in the cycle, we must have that $k+1$ immediately follows $s$. Since $\pi_k=1$, we must have $c_n = k$ and so $s,(k+1),k$ is a 321 occurrence in the cycle. Thus we must have $k \geq \frac{n+1}{2}$.

Since $\pi_k = 1$ and $\pi$ avoids 213, $\{\pi_{k+1}, \ldots, \pi_n\} = [2, n-k+1]$ with $n-k+1\leq k$. Also, because the cycle form avoids 321 and $c_n=k$, we must have that the elements of the cycle with values in $[k+1,n]$ appear in increasing order. Furthermore, these elements must each be followed by an element in $[2, n-k+1]$. Now, since $k+1$ is greater than any element in $[2, n-k+1]$, we must also have that the elements in $[2, n-k+1]$ appear in increasing order. Thus $\pi_i=i-k+1$ for all $i \in [k+1,n]$, and $C(\pi)$ contains the pairs $i, i-k+1$ for all $i \in [k+1, n]$ in increasing order. Furthermore, we claim that these pairs appear consecutively.  If $x\in [n-k+2,k-1]$ appears in $C(\pi)$ between the pair $i,i-k+1$ and $i+1, i-k+2$ for some $i \in [k+1, n-1]$, then $C(\pi)$ would contain the 321-pattern $i,x, i-k+2$. Thus these pairs appear in $C(\pi)$ consecutively and $\pi_i = k+i$ for all $i\in[2,n-k]$.

Suppose $\pi_r=k$. We now write $C(\pi)$ as
\begin{equation}\label{eqn:213-321} C(\pi) = (1, c_2, \ldots, c_s, k+1,2, k+2,3, k+3,4,\ldots, n, n-k+1,c_t,\ldots, c_{n-2}, r,k) \end{equation}
for some $s$ and $t$. We want to show $r \in \{k-2,k-1,n\}$. If $k < r < n$, then $nrk$ is a 321-pattern in $C(\pi)$.  Thus $r \in [n-k+2,k-1] \cup \{n\}$. Suppose toward a contradiction that $r < k-2$. The elements $k-1$ and $k$ must appear before $k+1$ in $C(\pi)$ otherwise $n(k-1)r$ or $nkr$ would be a 321-pattern.  Furthermore, the elements $c_2, c_3, \ldots, c_s$ must be increasing since they are followed by 2. Since $k-2$ and $k-1$ are the largest elements not already accounted for, this implies $c_{s-1} = k-2$ and $c_s = k-1$. But then in one-line notation, we have $\pi_r\pi_{k-2}\pi_{k-1} = k(k-1)(k+1)$ which is a 213 pattern. Therefore $r \in \{k-2,k-1,n\}$ as desired.
\end{proof}

\begin{lemma}\label{lem:213-321-3}
Suppose $n \geq 3$ and let $\pi \in \A_n(213;321)$ with $\pi_k = 1$ where $k\neq n$. Further suppose $\pi_r=k$.
\begin{itemize} 
\item If $r=n$, then  $n$ is odd, $k=\frac{n+1}{2}$, and $C(\pi) = (1, k+1,2, k+2, 3, \ldots, n,k)$.
\item If $r = k-2$, then $n$ is odd, $k > \frac{n+1}{2},$  and 
\[C(\pi) = (1, n-k+2, n-k+4, \ldots, k-3, k-1, k+1,2, k+2,3,\ldots, n, n-k+1, n-k+3, \ldots, k-4, k-2, k).\]
\end{itemize}
\end{lemma}

\begin{proof} If $\pi_n=k$, then by Equation (\ref{eqn:213-321}), $k=n-k+1$, or $k=\frac{n+1}{2}$. Thus $n$ is odd, and $C(\pi) = (1, k+1,2, k+2, 3, \ldots, n,k), $where $k=\frac{n+1}{2}$.

Suppose $r = k-2$ and write $C(\pi)$ as in Equation (\ref{eqn:213-321}):
\[ C(\pi) = (1, c_2, \ldots, c_s, k+1,2, k+2,3, k+3,4,\ldots, n, n-k+1,c_t,\ldots, c_{n-2}, k-2,k).\] Note that $n-k+1 < k$ since $n \neq k-2$, and thus $k > \frac{n+1}{2}$. Also, since $C(\pi)$ avoids 321, the elements $c_2, \ldots, c_s$ must be increasing as well as the elements $c_t, \ldots c_{n-2}, k-2,k$. Thus, $c_s - k-1$. In one-line notation, we then have
\[\pi = \pi_1 (k+2)(k+3)\ldots (n-1)n\pi_{n-k+1} \ldots \pi_{k-3}k(k+1)12\ldots (n-k+1).\]
Now, since $\pi_i \in [n-k+2, k-1]$ for $i \in \{1\}\cup[n-k+1,k-3]$, and $\pi$ avoids 213, we must have that $\pi_1 = n-k+2$ and $\pi_i = i+2$ for $i \in [n-k+1,k-3]$. If $n$ is odd, then this permutation is cyclic and has form: 
\[C(\pi) = (1, n-k+2, n-k+4, \ldots, k-3, k-1, k+1,2, k+2,3,k+1, 4,\ldots, n, n-k+1, n-k+3, \ldots, k-4, k-2, k).\]
If $n$ is even, then $(1, n-k+2, n-k+4, \ldots, k-4, k-2, k)$ forms a cycle of length $k-\frac{n}{2}+1$, which is less than $n$ when $k> \frac{n+1}{2}$. 
\end{proof}

\begin{lemma}\label{lem:213-321-4} The number of permutations $\pi \in \A_n(213;321)$ with $\pi_k=1$ where $k \neq n$ and $\pi_{k-1} = k$ is equal to $a_{n-1}(213;321) - \lceil \frac{(n-4)^2}{2} \rceil-1$.
\end{lemma}

\begin{proof}
By Lemma~\ref{lem:213-321-1}, the number of permutations $\pi' \in \A_{n-1}(213;321)$ with $\pi'_{n-1} \neq 1$ is $a_{n-1}(213;321) - \lceil \frac{(n-4)^2}{2} \rceil-1$. We will find a bijective correspondence between permutations $\pi \in \A_n(213;321)$ with $\pi_k=1$ where $k \neq n$ and $\pi_{k-1} = k$  and permutations $\pi' \in \A_{n-1}(213;321)$ with $\pi'_{n-1} \neq 1$.

Suppose first $\pi \in \A_n(213;321)$ with $\pi_k=1$ where $k \neq n$ and $\pi_{k-1} = k$. Form $\pi'$ by deleting $k$ from $\pi$. Equivalently, $C(\pi')$ is formed from deleting $k$ from $C(\pi)$. Since $\pi'$ is cyclic and continues to avoid 213 in its one-line form and 321 in its cycle form, we have $\pi' \in A_{n-1}(213;321)$. Notice that $\pi'_{k-1} = 1$.

Consider this process in reverse.  Suppose $\pi' \in \A_{n-1}(213;321)$ with $\pi'_{k-1} = 1$ where $k-1 \neq n-1$.  Form $\pi$ by inserting  $k$ just after $k-1$ in the cycle notation, or equivalently, insert $k$ in position $k-1$ in one-line notation. Clearly this will not introduce a 321 pattern to the cycle and thus $C(\pi)$ avoids 321. By Lemma~\ref{lem:213-321-2}, if $\pi'_r = k-1$, then $r \in \{k-3,k-2, n-1\}$. In other words, in the one-line notation of $\pi'$, $k-1$ must appear in position $k-3,$ $k-2$, or $n$. We examine each of these cases separately to show that inserting $k$ in position $k-1$ does not introduce a 213-pattern. If $k-1$ is in position $k-3$ of $\pi'$, inserting $k$ will only introduce a 213 pattern in $\pi_{k-2} < k-1$. However, from Lemma~\ref{lem:213-321-3}, $\pi'_{k-2} = k$ so no 213 pattern is introduced.  In the cases where $k-1$ is in position $k-2$ or $n-1$ of $\pi'$, clearly no 213 pattern can be introduced. Thus $\pi \in \A_n(213;321)$ and our result holds.
\end{proof}

\begin{theorem}\label{thm:213-321}
For $n\geq 4$. Then: 
\[a_n(213;321) = \begin{cases} a_{n-1}(213;321) + n-3  & \text{if $n$ is even} \\ a_{n-1}(213;321) + n-3+\frac{n-3}{2}& \text{if $n$ is odd},\end{cases}\]
which has closed form 
\[
a_n = \binom{n-2}{2} + \binom{\lceil\frac{n-2}{2}\rceil}{2}+ 2.
\]
\end{theorem}

\begin{proof}
First note that there are $\lceil \frac{(n-3)^2}{2} \rceil + 1$ permutations with $\pi_n=1$ by Lemma~\ref{lem:213-321-1}. Now, let us count the number with $\pi_n\neq 1$, or equivalently, $c_n \neq n$. The number of permutations $\pi \in \A_n(231;321)$ with $c_{n-1} = n$ is 1 when $n$ is odd and $0$ when $n$ is even by Lemma~\ref{lem:213-321-3}.  To count the permutations $\pi \in \A_n(231;321)$ with $c_n =k \neq n$ and $c_{n-1} = k-2$, we first see that $\frac{n+1}{2} < k < n$ by Lemmas~\ref{lem:213-321-2} and \ref{lem:213-321-3}, and thus there are $\frac{n-3}{2}$ possible choices for $k$. By Lemma~\ref{lem:213-321-3}, these choices each yield a unique permutation in $\A_n(213;321)$ when $n$ is odd and there are no possible choices for $k$ when $n$ is even. Finally, there are $a_{n-1}(213;321) - \lceil \frac{(n-4)^2}{2} \rceil-1$ permutations $\pi \in \A_n(213;321)$ with $c_n = k \neq n$ and $c_{n-1} = k-1$ by Lemma~\ref{lem:213-321-4}.

Taken together, we have 
\[a_n(213;321) = \begin{cases} a_{n-1}(213;321) + \lceil \frac{(n-3)^2}{2} \rceil - \lceil \frac{(n-4)^2}{2} \rceil & \text{if $n$ is even} \\ a_{n-1}(213;321) + \frac{n-1}{2} +\lceil \frac{(n-3)^2}{2} \rceil - \lceil \frac{(n-4)^2}{2} \rceil & \text{if $n$ is odd},\end{cases}\]
which is equivalent to the recurrence in the theorem statement.
\end{proof}

\begin{example} Consider $\A_7(213;321)$. The number of permutations in $\A_7(213;321)$ whose cycles end in $7$ in $\lceil \frac{(7-3)^2}{2}\rceil + 1 = 9$ by Lemma~\ref{lem:213-321-1}. These are found by taking the reverse-complement-inverse of the permutations in $\A_7(132;321)$ and are:
\[ 3745621, 3457621, 4675231, 2456731, 2475631, 2567341, 2357641, 2345671, 2346751.\]
There is exactly one permutation in $\A_7(213;321)$ with $c_7 =k \neq 7$ and $c_{6} = 7$ as given by Lemma~\ref{lem:213-321-3}:
\[ (1,5,2,6,3,7,4).\]
There are 2 permutations with $c_7 \neq 7$ and $c_6 = c_7-2$ given by $k \in \{5,6\}$ in Lemma~\ref{lem:213-321-3}:
\[ (1, 4,6,2,7,3,5), (1,3,5,7,2,4,6).\]
Finally, there are $a_6(213;321) - \lceil \frac{(7-4)^2}{2} \rceil -1 = 9-5-1=3$ permutations with $\pi_7 \neq 1$ and $c_{6} = c_7 - 1$. We form these by starting with any permutation in $\A_6(213;321)$ with $\pi_{k-1} = 1$ and $k-1 < 6$, and insert $k$ after $k-1$ in cycle notation. The desired permutations in $\A_6(213;321)$ are $(1,6,2,3,4,5)$, $(1,3,6,2,4,5)$, and $(1,5,2,6,3,4)$. Inserting the appropriate element in the cycle structure yields the following permutations in $\A_7(213;321)$:
\[ (1,7,2,3,4,5,6), (1,3,7,2,4,5,6), (1,6,2,7,3,4,5).\]
\end{example}

\subsection{$\A_n(231;321)$}

This is the simplest case for $\tau=321$. 

\begin{theorem}\label{thm:231-321}
For $1\leq n\leq 4$, we have $a_n(231;321)=1$ and for $n\geq 5$, $a_n(231;321)=0$. 
\end{theorem}
\begin{proof}
First notice that if $\pi$ is cyclic and avoids $\sigma=231$, we must have $\pi_1=n$. Indeed, if we suppose $\pi_1=j$, then $\pi = j\pi_2\pi_3\ldots \pi_j\pi_{j+1}\ldots \pi_n$ with $\{\pi_2,\pi_3,\ldots,\pi_j\} = [1,j-1]$ and $\{\pi_{j+1},\ldots,\pi_n\}=[j+1,n]$. However, since the numbers in $[j]$ map to themselves, this is only cyclic if $j=n$.

Now, since $\pi_1=n$, the cycle form of $\pi$ is $C(\pi) = (1,n,c_3,\ldots, c_n)$. Since $C(\pi)$ avoids 321, we must have $C(\pi) = (1,n,2,3,4,\ldots, n-1)$ and so $\pi = n34\ldots (n-1)12$, which only avoids 231 if $n\leq 4$. 
\end{proof}

\subsection{$\A_n(312;321)$}

Let us start with a few lemmas.

\begin{lemma}\label{lem:312-321-1}
Let $n \geq 2$. For any permutation $\pi \in \A_n(312;321)$, we must have $\pi_n =1$ and $\pi_1 \in \{2,3,4\}$.
\end{lemma}
\begin{proof}
First note that if $\pi$ avoids 312 in its one line notation, it is of the form $\pi = \pi_1\pi_2\ldots \pi_{k-1} 1 \pi_{k+1}\ldots \pi_n$ for some $k \in [2,n]$ with $\{\pi_1, \ldots, \pi_k\} = [1,k]$ and $\{\pi_{k+1},\ldots, \pi_n\} = [k+1,n]$. In particular, the elements $[1,k]$ map to themselves under the permutation and so if $k<n$, $\pi$ is composed of at least two cycles. Therefore $\pi_n=1$.

If $n \leq 5$, it is easy to check that the results hold, so for the remainder of the proof suppose $n \geq 6$. For the sake of contradiction, let us suppose that $\pi_1>4$. Then since the cycle avoids $321$, we must have that 2, 3, and 4, appear in the cycle in increasing order. Also, we must have that $\pi_2 \neq 3$ because otherwise $\pi_134$ would be a 312 occurrence in $\pi$. We can then write $\pi = \pi_1\pi_2 \pi_3 \pi_4\ldots \pi_{m-1} 2 \pi_{m+1} \ldots \pi_{n-1} 1$ for some $m \in [5,n-1]$. (We must have $m \geq 5$ since $\pi_4 \neq 2$.) Since $\pi$ avoids $312$, it must be the case that $\{\pi_1, \ldots, \pi_{m-1}\} = [3,m+1]$.

Notice $\pi_3 \notin \{1,2,3\}$, so we consider the two cases where $\pi_3=4$ or $\pi_3 > 4$. If $\pi_3 = 4$, then $C(\pi)$ contains the pattern $m\pi_23$ thus implying $\pi_2 = m+1$. But then $(m+1)4\pi_4$ is a 312 pattern in $\pi$ which is a contradiction.  On the other hand, if $\pi_3>4$,  at least one of $\pi_2$ or $\pi_4$ is less than $m$. Then either $m\pi_24$ or $m\pi_34$ is a 321 pattern in $C(\pi)$. We have thus proven that $\pi_1\leq 4$.
\end{proof}

\begin{lemma}\label{lem:312-321-2}
Suppose $n \geq 4$ and $\pi \in \A_n(312;321)$. If $\pi_1 = 3$, then either:
\begin{itemize}
\item $\pi_2=4$, or
\item $\pi_2=5$ and $\pi_3=4$.
\end{itemize}
\end{lemma}

\begin{proof}

Choose $m \in [3,n-1]$ so that $\pi_m = 2$. Thus $\pi = 3\pi_2\cdots \pi_{m-1} 2 \pi_{m+1} \cdots \pi_{n-1}1$ with $\{\pi_1, \ldots, \pi_{m-1}\} =[3,m+1]$. For reference, $C(\pi) = (1,3, \pi_3, \ldots, m,2, \pi_2, \ldots, c_{n-1},n)$. Note that $\pi_2 \neq m$ so $\pi_2 \in \{m+1\} \cup [4,m-1]$. Let us first show that $\pi_2\in\{4,5\}$.

Suppose $\pi_2 \in [6,m-1]$. Then in $C(\pi)$, we must have 4 and 5 appear before $2$ in the cycle or else  $m, \pi_2, 4$ or $m, \pi_2, 5$ would be a 321 pattern. Also, since the segment of the cycle appearing before 2 is increasing, we must have $\pi_3=4$ and $\pi_4=5$. But then $\pi_245$ is a 312 pattern in $\pi$ which is a contradiction. Now suppose $\pi_2 = m+1$.Since $\pi$ avoids 312, we must have $\pi_3=m$ otherwise $(m+1)\pi_3m$ is a 312 pattern.   But then $\{\pi_4, \pi_5, \ldots, \pi_{m-1}\} = [4, m-1]$, implying $\pi$ is not cyclic. Therefore we have $\pi_2\in\{4,5\}$.

Finally, let us show that if $\pi_1 =3$ and $\pi_2 = 5$, then we must have $\pi_3 = 4.$  First, note that $\pi_3\neq 2$ since other wise $524$ would appear as a 312 pattern in $\pi$. Now, notice that in the cycle, we have $C(\pi) = (1, 3, \pi_3, \ldots, m,  2, 5,  \ldots, c_{n-1}, n)$. If $4$ appears before 2 in the cycle, we must have that $\pi_3=4$ since $C(\pi)$ avoids 321. On the other hand, if $4$ appears after 2 in the cycle, then $\pi_3>5$, and so, $\pi_3, 5, 4$ will be a 321 pattern in $C(\pi)$.
\end{proof}

\begin{lemma}\label{lem:312-321-3}
Suppose $n \geq 5$ and $\pi \in \A_n(312;321)$. If $\pi_1 = 4$, then $\pi_2 = 3$ and either:
\begin{itemize}
\item $\pi_3=5$, or
\item $\pi_3=6$ and $\pi_4=5$.
\end{itemize}
\end{lemma}
\begin{proof}
Since $\pi$ avoids 312 and $\pi_1=4$, we must have $\pi = 4\pi_2 \ldots \pi_{s-1} 3 \pi_{s+1} \ldots \pi_{t-1}2 \pi_{t+1} \ldots \pi_{n-1}1$ for some $s,t \in [2,n-1]$ with $s < t$. Let us first show that $\pi_2=3$. If we assume $\pi_2\neq 3$, we must have $\pi_2>4$. In the cycle we would have $C(\pi) = (1, 4, \pi_4, \ldots, t, 2, \pi_2, \ldots, s, 3, \ldots, n)$ where $s\neq 2$. Thus $t, s, 3$ would be a 321 pattern in $C(\pi)$. Therefore, $\pi_2=3$.

Next let us see that if $\pi_1=4$ and $\pi_2=3$, then $\pi_3\in\{5,6\}$. We know that $C(\pi) = (1, 4, \pi_4, \ldots, t, 2,3,\pi_3, \ldots, n)$ and $\pi = 43\pi_3\pi_4 \ldots \pi_{t-1}2 \pi_{t+1} \ldots \pi_{n-1}1$ with $\{\pi_3, \ldots, \pi_{t-1}\} = [5,t+1]$. For the sake of contradiction, suppose $\pi_3 \in [7,t+1]$. Note that if $5$ and $6$ both appear before $2$ in the cycle notation, we must have $\pi_4=5$ and $\pi_5=6$ since $C(\pi)$ avoids 321. However, then we would have $\pi_356$ as a 312 pattern in $\pi$. Thus, at least one of $5$ or 6 appears after the 2 in the cycle notation. Now, if $\pi_3<t$, then $t,\pi_3,5$ or $t,\pi_3,6$ would be a 321 pattern in $C(\pi)$. Therefore, it must be that $\pi_3=t+1$. The remaining entries before 2 in $\pi$ must be increasing since they are all less than $t+1$.  In particular, $\pi_4=t$ and $\pi_5=t-1$, and $C(\pi) = (1, 4, t, 2,3, t+1, \ldots, 5, t-1, \ldots, n)$. If $t\geq8$, then this implies 6 appears after the 5, and so we have that $t, t-1, 6$ is a 321 pattern. If $t=7$, then $\pi = 4387652 \pi_{t+1} \ldots \pi_{n-1}1$, so $\pi_5=6$ and $\pi_6=5$, so $\pi$ is not cyclic. If $t=6,$ we have $\pi = 437652 \pi_{t+1} \ldots \pi_{n-1}1$ and so $5$ is a fixed point and thus $\pi$ is not cyclic. Therefore, we have shown that $\pi_3 \in \{5,6\}$.

Finally, we will show that if $\pi_1=4$, $\pi_2=3$, and $\pi_3=6$, then $\pi_4=5$. In this case, we know that $C(\pi) = (1, 4, \pi_4, \ldots, t, 2,3, 6, \pi_6 \ldots, n)$ and $\pi = 436\pi_4 \ldots \pi_{t-1}2 \pi_{t+1} \ldots \pi_{n-1}1$. Notice first that $\pi_4\neq 2$ since in that case, 625 would be a 312 pattern in $\pi$. Therefore, $\pi_4\in\{5\}\cup [7,t+1]$. If the $5$ appears after the 2 in the cycle-notation, then $\pi_4>6$, so $\pi_4, 6, 5$ would be a 321 pattern. If the $5$ appears before the 2 in the cycle notation, it must be that $\pi_4=5$ since otherwise $\pi_4,5,2,$ is a 321 pattern.
\end{proof}

\begin{theorem}\label{thm:312-321}
For $n\geq 7$, the number of permutations in $\A_n(312;321)$ satisfies the recurrence \[
a_n(312;321) = a_{n-1}(312;321)+a_{n-2}(312;321)+2a_{n-3}(312;321)+a_{n-4}(312;321),
\]
where $a_3(312;321) = 1$, $a_4(312;321) = 2$, $a_5(312;321) = 5$, and $a_6(312;321) = 10$. In closed form,
\[ a_n = \sum_{k=0}^{\lfloor \frac{n-2}{2} \rfloor}\  \sum_{j=0}^{n-2-2k} {n-2-k-j \choose k}{2k \choose j}\]
for $n \geq 3$.
\end{theorem}

\begin{proof}
First, we show there are $a_{n-1}(312;321)$ permutations that have $\pi_1=2$. In this case, all permutations in $\A_n(312;321)$ with $\pi_1 = 2$ can be obtained from permutations $\pi' \in \A_{n-1}(312;321)$ by inserting 2 in the first position of the one-line notation $\pi'$, or equivalently inserting a 2 immediately after 1 in $C(\pi')$. In this case, we clearly do not introduce a 312 pattern in $\pi$ or a 321 pattern in $C(\pi)$. The process is reversible and the correspondence is bijective.

Next, we will see that there are $a_{n-2}(312;321)$ permutations that have $\pi_1=3$ and $\pi_2=4$. In this case, all permutations in $\A_n(312;321)$ with $\pi_1 = 3$ and $\pi_2=4$ can be obtained from permutations $\pi' \in \A_{n-2}(312;321)$ by inserting 3 in the first position and $4$ in the second position of the one-line notation $\pi'$, or equivalently inserting a 3 immediately after 1 and a 4 immediately after the 2 in $C(\pi')$. In this case, we do not introduce a 312 pattern in $\pi$ or a 321 pattern in $C(\pi)$. The process is invertible so the correspondence is bijective.

\sloppypar{We now consider permutations in $\A_n(312;321)$ that have $\pi_1\pi_2\pi_3=354$; we will show there are $a_{n-3}(312;321)$ such permutations. In this case, any permutations $\pi \in \A_n(312;321)$ with $\pi_1\pi_2\pi_3=345$ can be obtained from some permutation $\pi' \in \A_{n-3}(312;321)$ by inserting 3 in the first position, $5$ in the second position, and 4 in the third position of the one-line notation $\pi'$, or equivalently inserting a 34 immediately after 1 and a 5 immediately after the 2 in $C(\pi')$. In this case, we do not introduce a 312 pattern in $\pi$ or a 321 pattern in $C(\pi)$.}

Next, we will see that there are $a_{n-3}(312;321)$ permutations that have $\pi_1\pi_2\pi_3=435$. In this case, all permutations in $\A_n(312;321)$ with $\pi_1\pi_2\pi_3=435$ can be obtained from permutations $\pi'=\A_{n-3}(312;321)$ by inserting 4 in the first position, $3$ in the second position, and 5 in the third position, or equivalently inserting a 4 immediately after 1 and a 35 immediately after the 2 in $C(\pi')$. In this case, we do not introduce a 312 pattern in $\pi$ or a 321 pattern in $C(\pi)$.

Finally, we will see that there are $a_{n-4}(312;321)$ permutations that have $\pi_1\pi_2\pi_3\pi_4=4365$. In this case, all permutations in $\A_n(312;321)$ with $\pi_1\pi_2\pi_3\pi_4=4365$ can be obtained from permutations $\pi'=\A_{n-4}(312;321)$ by inserting 4 in the first position, $3$ in the second position, 5 in the third position, and 5 in the fourth position, or equivalently inserting a 45 immediately after 1 and a 36 immediately after the 2 in $C(\pi')$. In this case, we do not introduce a 312 pattern in $\pi$ or a 321 pattern in $C(\pi)$.
\end{proof}

\begin{example} Consider $\A_7(312;321)$. We build this set recursively by first listing the permutations in $\A_3(312;321)$, $\A_4(312;321)$, $\A_5(312;321)$, and $\A_6(312;321)$:
\begin{align*}
\A_3(312;321) &= \{231\},\\
\A_4(312;321) &= \{3421, 2341\},\\
\A_5(312;321) &= \{24531, 23451, 34251, 43521, 35421\}, \\
\A_6(312;321) &= \{235641, 234561, 245361, 254631, 246531, 342561, 435261, 354261, 345621, 436521\}.
\end{align*}
To obtain the permutations in $\A_7(312;321)$, insert 4365 in the front of the permutation in $\A_3(312;321)$, insert 354 and 435 in the front of the permutations in $\A_4(312;321)$, insert 34 in the front of the permutations in $\A_5(312;321)$, and insert 2 in the front of the permutations in $\A_6(312;321)$:
\begin{align*}
\A_7(312;321) &= \{4365271, 3546721, 3542671, 4356721, 4352671,3426751, 3425671, 3456271, 3465721, 3457621,\\
& \qquad 2346751, 2345671, 2356471, 2365741, 2357641, 2453671, 2546371, 2465371, 2456731, 2547631\}.
\end{align*}
\end{example}

\subsection{$\A_n(321;321)$}

We will count these by the position of $n$.
\begin{lemma}\label{lem:321-321-1} Let $n \geq 2$. Then the number of permutations $\pi \in \A_n(321;321)$ with $\pi_{n-1} = n$ is $a_{n-1}(321;321)$.
\end{lemma}

\begin{proof} Let $\pi \in \A_n(321;321)$ with $\pi_{n-1} =n$. Form $\pi'$ by deleting $n$ from $\pi$. Equivalently, $C(\pi')$ is formed from $C(\pi)$ by deleting $n$. Thus $\pi' \in \A_{n-1}(321;321)$. Now consider the process in reverse by letting $\pi' \in \A_{n-1}(321;321)$. Form $\pi$ by inserting $n$ in position $n-1$ of $\pi'$, or equivalently, inserting $n$ after $n-1$ in $C(\pi')$. Since $\pi \in \A_n(321;321)$ with $\pi_{n-1} =n$, our result holds.
\end{proof}

\begin{lemma}\label{lem:321-321-2} Let $n \geq 3$. Then the number of permutations $\pi \in \A_n(321;321)$ with $\pi_{n-2} = n$ is $\lfloor \frac{n-1}{2} \rfloor$.
\end{lemma}

\begin{proof}
Let $\pi \in \A_n(321;321)$ with $\pi_{n-2} = n$ and consider the position of 1 in one-line notation.  We note that $\pi_n \neq 1$ since otherwise $n\pi_{n-1}1$ is a 321 pattern. Also, if $\pi_j=1$ for some $j \in [2,n-3]$, then $j$ is the last element in cycle notation and thus $n-1$ must come before $n$. However, then $(n-1)(n-2)j$ is a 321 pattern in $C(\pi)$ and thus $\pi_{n-1} = 1$.

In one-line notation, all elements before 1 must be increasing.  In particular, if $\pi_n=k$ for some $k \in [2, n-1]$, then $\pi_i = i+1$ for $i \in [1, k-2]$ and $\pi_i = i+2$ for $i \in [k-1, n-3]$. However, if $k$ is the same parity as $n$, then the cycle form of $\pi$ contains the cycle $(1,2,3, \ldots, k-1, k+1, k+3, \ldots, n-1)$ which contradicts the fact that $\pi$ is cyclic. Thus, $k$ must be the opposite parity as $n$. In this case
\begin{equation}\label{eqn:321-321-1}
C(\pi) = (1, 2, 3, \ldots, k-1, k+1, k+3, \ldots, n, k, k+2, k+4, \ldots, n-1)
\end{equation}
which avoids $321$. Thus if $k \in [2, n-1]$ and $k$ is the opposite parity as $n$ there is exactly one permutation yielding a total of $\lfloor \frac{n-1}{2} \rfloor$ such permutations.
\end{proof}

\begin{lemma}\label{lem:321-321-3} Let $n \geq 5$, $k \in [2,n-3]$ and $\pi \in \A_n(321;321)$ with $\pi_k = n$. Then
\[ C(\pi) = (1, c_2, c_3, \ldots, c_{n-5}, k-1, n-1, k, n, k+1).\]
\end{lemma}

\begin{proof} It is easy to verify this is true for $n=5$, so assume $n \geq 6$. Let $\pi$ be as given and consider the position of 1 in one-line notation. If $\pi_j=1$ for some $j$, then $j \in [2,k-1] \cup \{k+1\}$ since otherwise $n \pi_{k+1}1$ is a 321 pattern.  Suppose $j \in [2,k-1]$. Then in cycle notation, $k > j$ and $k$ comes before $j$.  So the elements in $[k+1, n-1]$ must appear after $n$ in increasing order to avoid 321. Thus $\pi_{n-2} = n-1$ and $\pi_{n-1} = j$. But then $n(n-1)j$ is a 321-pattern in one-line and it must be that $\pi_{k+1} = 1$.

Note that in $C(\pi)$, elements in $[k+2,n-1]$ must come before $n$ since $C(\pi)$ ends in $k+1$. They must also appear in increasing order because $k$ appears after them.  However, $k+2$ comes before $k$ in $C(\pi)$ and thus all elements after $k$, excepting $k+1$, must be bigger than $k+2$.  Thus only $k+1$ appears after $n$ and we have $\pi_n=k+1$ as desired.

We now consider the one-line notation of $\pi$ again.  Since $\pi$ avoids 321, all elements after $n$ must appear in increasing order and all elements before 1 must appear in increasing order.  If $n-1$ appears after $n$, it must appear in position $n-1$ which contradicts the fact that $\pi$ is cyclic.  Thus $n-1$ appears before $n$ and must occur in position $k-1$.

Finally, consider $\pi_{n-1}$ which must be less than $k+1$ since otherwise $n\pi_{n-1}(k+1)$ would be a 321 pattern in the one-line notation. Furthermore, since $\pi_{k-1} = n-1,$ we have $\pi_{n-1} \neq k-1$ (since $\pi$ is cyclic), and thus $\pi_{n-1} \in [2,k-2] \cup \{k\}$. However, if $\pi_{n-1} \leq k-2$, then in cycle notation we have the pattern $n(k-1)(k-2)$ which is a contradiction. Thus $\pi_{n-1} = k$ and our results hold.
\end{proof}

\begin{lemma} \label{lem:321-321-4} Let $n \geq 5$. The number of permutations $\pi \in \A_n(321;321)$ with $\pi_{n-1} \neq n$ is ${\lceil n/2 \rceil \choose 2}$.
\end{lemma}

\begin{proof} We first show that for $k \in [2,n-3]$, the number of permutations  $\pi \in \A_n(321;321)$ with $\pi_{k} = n$ is equal to the number of permutations $\pi' \in \A_{n-2}(321;321)$ with $\pi'_{k-1} = n-2$.  Let $\pi \in \A_n(321;321)$ with $\pi_{k} = n$. By Lemma~\ref{lem:321-321-3}, 
\[ C(\pi) = (1, c_2, c_3, \ldots, c_{n-5}, k-1, n-1, k, n, k+1).\] Consider the permutation $\pi'$ formed by deleting $n$ and $k+1$ from $\pi$. Then
\[ C(\pi') = (1, c_2', c_3', \ldots, c_{n-5}', k-1, n-2, k)\] where $c_i' =c_i$ if $i \leq k$ and $c_i' = c_i - 1$ if $i > k$. Thus $\pi' \in \A_n(321;321)$ and has the additional property that $\pi_{k-1}' = n-2$. Conversely, if $\pi' \in \A_{n-2}(321;321)$ with $\pi'_{k-1} = n-2$, form $\pi$ by inserting $n$ in position $k$ followed by $k+1$ at the end.

Now the total number of permutations in $\A_n(321;321)$ that do not not have $n$ in the penultimate position can be found using induction. The base case $n=5$ is easily checked so assume $n > 5$. The number of permutations $\pi \in \A_n(321;321)$ with $\pi_{n-1} \neq n$ is equal to the number with $\pi_{n-2}$ plus the number with $\pi_{k} = n$ for $k \in [2,n-3]$. Using Lemma~\ref{lem:321-321-2} and induction, we have the total number is given by
\[ \left\lfloor \frac{n-1}{2} \right\rfloor + {\lceil \frac{n-2}{2} \rceil \choose 2} = {\lceil \frac{n}{2} \rceil \choose 2}.\] \end{proof}

Taking this fact along with the results in Lemma~\ref{lem:321-321-1}, we can enumerate $\A_n(321;321)$.

\begin{theorem}\label{thm:321-321} For $n \geq 5$, the number of permutations in $\A_n(321;321)$ satisfies the recurrence 
\[ a_n(321;321) = a_{n-1}(321;321) + {\lceil \frac{n}{2} \rceil \choose 2}\]
where $a_3(321;321) = 2$ and $a_4(321;321) = 3$. In closed form,
\[ a_n(321;321) = \begin{cases} \displaystyle{1 + 2{\frac{n+2}{2} \choose 3}} & \text{if } n \text{ is even},\\ 
\displaystyle{1 + 2{\frac{n+1}{2} \choose 3} + { \frac{n+1}{2} \choose 2}} & \text{if } n \text{ is odd}. \end{cases} \]
\end{theorem}

\begin{proof} Lemmas~\ref{lem:321-321-1} and \ref{lem:321-321-4} combine to give the recurrence relation.  Solving the recurrence relation gives the provided closed form.
\end{proof}

\begin{example} Consider $\A_6(321;321)$. Because some of these permutations are found recursively, we first note that
\[ A_5(321;321) = 23514, 24153, 31452, 23451, 34512, 45123\}.\]
The permutations in $\A_6(321;321)$ that have 6 in the penultimate position are found recursively by inserting a 6 in position 5 of all those permutations in $\A_5(321;321)$. Thus $\A_6(321;321)$ includes the following permutations:
\[  235164, 241563, 314562, 234561, 345162, 451263.\]
The proof of Lemma~\ref{lem:321-321-2} gives all $\lfloor \frac{6-1}{2} \rfloor = 2$ permutations in $\A_6(321;321)$ where $\pi_4=6$. These 2 permutations are formed from Equation~(\ref{eqn:321-321-1}) where $k \in [2,5]$ and $k$ is the opposite parity as 6. Thus, using $k\in \{3,5\}$ in Equation~(\ref{eqn:321-321-1}), we have the following additional 2 permutations in $\A_6(321;321)$:
\[ (1,2,4,6,3,5), (1,2,3,4,6,5).\]
By the proof of Lemma~\ref{lem:321-321-4}, the permutations in $\A_6(321;321)$ where $\pi_k=6$ for $k \in \{2,3\}$ are formed recursively by considering all permutations in $\pi' \in \A_4(321;321)$ with $\pi'_{k-1}$ = 4. The new permutations are formed by inserting 6 in position $k$ followed by $k$ at the end. When $k=2$, there are no permutations in $\A_4(321;321)$ with 4 in position 1. For $k=3$, there is exactly one permutation in $\A_4(321;321)$ with 4 in position 2, namely 2413.  By inserting 6 in position 3 and 4 at the end, we get the following permutation in $\A_6(321;321)$: 256134.

Thus we have ${3 \choose 2} = 3$ permutations in $\A_6(321;321)$ that do not have 6 in position 5, and $a_5(321;321) = 6$ permutations that have 6 in position 5 for a total of 9 permutations. We note that $1 + 2{4 \choose 3} = 9$ and the closed form holds for $n=6$.
\end{example}

\section{Further directions for research}

There are several directions of future research to consider, including avoidance of longer patterns. For example, we conjecture that $a_n(3412;213)$ has generating function \[
 \frac{2x}{1-2x+\sqrt{1-4x+4x^3}}
\] associated to the OEIS sequence A087626. We similarly conjecture that $a_n(1324,1423; 213) = \binom{n}{3}+1$ and $a_n(3421,4321;213) = F_{2n-3}$. By considering other permutations or sets of permutations, many other interesting sequences seem to appear. A few results involving longer patterns can be found in the follow-up paper to this one \cite{ABBGJ-2}, in which the authors extend Theorems~\ref{thm:321-213} and \ref{thm:321-231}, finding the number of cyclic permutations $\pi$ that avoid the monotone decreasing permutation of length $k$ and where $C(\pi)$ avoids $\tau$ for $\tau \in \{213,231,312\}$. 

One could also consider cyclic permutations where all cycle forms (instead of just the cycle form beginning with 1) avoid a given pattern. A few results in this vein can be found in  the follow-up paper  \cite{ABBGJ-2}, enumerating cyclic permutations $\pi$ that avoid the monotone decreasing permutation of length $k$ where all cycle forms avoid a pattern of length 4. 
Finally, it would additionally be interesting to consider avoidance among other cycle types. Since originally written, some of the results in this paper have been extended to other cycle types in \cite{AL24} by considering permutations so that $\pi$ and its image under the fundamental bijection both avoid a given pattern. 

\subsection*{Acknowledgements}

The student authors on this paper, Ethan Borsh, Jensen Bridges, and Millie Jeske, were funded as part of an REU at the University of Texas at Tyler sponsored by the NSF Grant DMS-2149921.



\bibliographystyle{amsplain}

\begin{thebibliography}{99}

\bibitem{A22} K. Archer, Enumerating two permutation classes by number of cycles, 
\textit{Discrete Mathematics and Theoretical Computer Science}. 22(2) (2022). 


\bibitem{ABBGJ-2} K. Archer, E. Borsh, J. Bridges, C. Graves, M. Jeske. Pattern-restricted cyclic permutations with a pattern-restricted cycle form. Preprint arXiv:2408.15000.

\bibitem{AE14} K. Archer and S. Elizalde, Cyclic permutations realized by signed shifts, \textit{J. Comb.} 5(1) (2014), 1--30. 



\bibitem{AG21} K. Archer and C. Graves, Pattern-restricted permutations composed of 3-cycles, \textit{Discrete Mathematics}, 345(7) (2022). 112895.

\bibitem{AL20}  K. Archer and L.-K. Lauderdale, Enumeration of cyclic permutations in vector grid classes, \textit{Journal of Combinatorics}, 11(1) (2020), 203--230.

\bibitem{AL17} K. Archer and L.-K. Lauderdale, Unimodal permutations and almost-increasing cycles. \textit{Electronic Journal of Combinatorics}, 24(3) (2017), \#P3.36.


\bibitem{AL24} K. Archer and R. Laudone, Pattern avoidance and the fundamental bijection, Preprint arXiv:2407.06338.


\bibitem{BT23} Y. Berman and B. Tenner, Pattern-functions, statistics, and shallow permutations, \textit{Electronic Journal of Combinatorics} 29(4), (2022), \#P4.43.

\bibitem{BC19}  M. B\'{o}na and M. Cory, Cyclic permutations avoiding pairs of patterns of length three. \textit{Discrete Math. Theor. Comput. Sci.} 21 (2019), no. 2, Paper No. 8, 15 pp.




\bibitem{BS19} M. B\'{o}na and R. Smith, Pattern avoidance in permutations and their squares, \textit{Discrete Math.} 342(11) (2019), 3194--3200. 

\bibitem{BD19} A. Burcroff and C. Defant, Pattern-avoiding permutation powers, \textit{Discrete Math.}  343(11) (2020), 1--11.

\bibitem{DRS07} E. Deutsch, A. Robertson, and D. Saracino, Refined restricted involutions, \textit{European J. Combin.} 28(1) (2007), 481--498. 

\bibitem{DELMSSS} R. Domagalski, S. Elizalde, J. Liang, Q. Minnich, B. Sagan, J. Schmidt, A. Sietsema, Cyclic pattern containment and avoidance, \textit{Advances in Applied Math.} 135, (2022), 102320.

\bibitem{E11}
S.~Elizalde, The $\mathcal{X}$-class and almost-increasing permutations, {\em Ann. Comb.} 15, (2011), 51--68

\bibitem{ET18}
S.~Elizalde and J.~Troyka, Exact and asymptotic enumeration of cyclic permutations according to
  descent set, {\em J. Combin. Theory Ser. A}165, (2019), 360--391.

\bibitem{GR93}
I.~Gessel and C.~Reutenauer, Counting permutations with given cycle structure and descent set, {\em J. Combin. Theory Ser. A} 64 (1993), 189--215.


\bibitem{H19} B. Huang, An upper bound on the number of $(132,213)$-avoiding cyclic permutations. \textit{Discrete Math.} 342(6) (2019), p. 1762--1771.

\bibitem{GM02} O. Guibert and T. Mansour, Restricted 132-involutions, \textit{S\'{e}minaire Lotharingien de Combinatoire} 48 (2002), Article B48a. 

\bibitem{P23} J. Pan, On a conjecture about strong pattern avoidance, \textit{Graphs and Combinatorics} 39:2 (2023).

\bibitem{SS85} R. Simion and F. W. Schmidt, Restricted permutations, \textit{European J. Combin.} 6 (1985), 383--406. 

\bibitem{S07}
R.~Stanley, Alternating permutations and symmetric functions, {\em J. Combin. Theory Ser. A}, 114(3):436--460, 2007.

\bibitem{T22} B.~Tenner, Boolean intersection ideals of permutations in the Bruhat order, \textit{ECA} 2(3) (2022), Article \# S2R23.

\bibitem{V03} A. Vella, Pattern avoidance in permutations: linear and cyclic orders, \textit{Electronic J. Combin.}, 10 (2003) \#R18.

\end{thebibliography}

\end{document}